\documentclass{article}

 \usepackage[final,nonatbib]{neurips_2019}
\usepackage{graphicx}

\usepackage[utf8]{inputenc} 
\usepackage[T1]{fontenc}    
\usepackage{hyperref}       
\usepackage{url}            
\usepackage{booktabs}       
\usepackage{amsfonts}       
\usepackage{nicefrac}       
\usepackage{microtype}      

\usepackage{bm}
\usepackage{amsmath,amsthm}
\usepackage{mathtools}
\usepackage{algorithmic}
\usepackage[algoruled,boxed,lined]{algorithm2e}
\usepackage{color}
\usepackage{epstopdf}

\DeclarePairedDelimiter\floor{\lfloor}{\rfloor}
\DeclarePairedDelimiter\ceil{\lceil}{\rceil}
\newcommand{\R}{\mathbb R}

\newcommand{\A}{\mathcal A}

\newcommand{\I}{\mathcal I}

\newcommand{\B}{\mathcal B}

\DeclareMathOperator*{\argmin}{arg\,min}
\DeclareMathOperator*{\E}{\mathbb{E}}
\newcommand{\Q}{\mathcal Q}
\newcommand{\Rc}{\mathcal R}
\newcommand{\Pc}{\mathcal P}
\newcommand{\J}{\mathcal J}

\usepackage{subfig}
\usepackage{verbatim}
\newtheorem{assumption}{Assumption}

\newcommand{\mbf}{\mathbf}
\newtheorem{lemma}{Lemma}
\newtheorem{theorem}{Theorem}
\newtheorem{corollary}{Corollary}
\newtheorem{proposition}{Proposition}

\newtheorem{example}{Example}
\newtheorem{definition}{Definition}

\theoremstyle{remark}

\newtheorem{remark}{Remark}

\allowdisplaybreaks

\newcommand{\nbf}{\noindent\textbf}
\newcommand{\nit}{\noindent\textit}

\newcommand{\black}[1]{\textcolor{black}{#1}}

\title{Online Optimal Control with Linear Dynamics and Predictions: Algorithms and Regret Analysis}

\author{%
  Yingying Li \\
  SEAS\\
  Harvard University\\
 Cambridge, MA, 02138 \\
  \texttt{yingyingli@g.harvard.edu} \\
  \And
   Xin Chen \\
   SEAS \\
  Harvard University\\
  Cambridge, MA, 02138 \\
  \texttt{chen\_xin@g.harvard.edu}\\
   \And
   Na Li \\
   SEAS \\
  Harvard University\\
  Cambridge, MA, 02138 \\
   \texttt{nali@seas.harvard.edu} \\
}

\begin{document}

\maketitle
\begin{abstract}
	This paper studies the online optimal control problem with time-varying convex stage costs for a time-invariant linear dynamical system, where a finite lookahead window of accurate predictions of the stage costs are available at each time. We design online algorithms, Receding Horizon Gradient-based Control (RHGC), that utilize the predictions through finite steps of gradient computations. We study the algorithm performance measured by \textit{dynamic regret}: the online performance minus the optimal performance in hindsight. It is shown that the dynamic regret of RHGC decays exponentially with the size of the lookahead window. In addition, we provide a fundamental limit of the dynamic regret for any online algorithms by considering linear quadratic tracking problems. The regret upper bound of one RHGC method almost reaches the fundamental limit, demonstrating the effectiveness of the algorithm. Finally, we numerically test our algorithms for both linear and nonlinear systems to show the effectiveness and generality of our RHGC.
	 
\end{abstract}

\section{Introduction}

In this paper, we consider a $N$-horizon discrete-time sequential decision-making problem. At each time $t=0, \dots,N-1$, the decision maker observes a state $x_t$ of a dynamical system, receives a $W$-step lookahead window of future cost functions of states and control actions, i.e. $f_t(x)+g_t(u), \dots, f_{t+W-1}(x)+g_{t+W-1}(u)$, then decides the control input $u_t$ which drives the system to a new state $x_{t+1}$ following some known dynamics. For simplicity, we consider a linear time-invariant (LTI) system $x_{t+1}=Ax_t+Bu_t$ with $(A,B)$ known in advance. The goal is to minimize the overall cost over the $N$ time steps. This problem enjoys many applications in, e.g. data center management \cite{lazic2018data,xu2006predictive},  robotics \cite{baca2018model}, autonomous driving \cite{rios2016survey,kim2014mpc}, energy systems \cite{kouro2008model},  manufacturing \cite{perea2003model,wang2007model}.
Hence, there has been a growing interest on the problem, from both control and online optimization communities.

In the control community, studies on the above problem focus on economic model predictive control (EMPC), which is a variant of model predictive control (MPC) with a primary goal on optimizing economic costs \cite{diehl2010lyapunov,muller2017economic,ellis2014tutorial,ferramosca2010economic,ellis2014economic,angeli2016theoretical,amrit2011economic,grune2013economic}. 
Recent years have seen a lot of attention on the optimality performance analysis of EMPC, under both time-invariant costs \cite{angeli2012average,grune2014asymptotic,grune2015non} and time-varying costs \cite{ferramosca2014economic,ferramosca2010economic,angeli2016theoretical,grune2017closed,grune2018economic}. However, most studies focus on asymptotic performance and there is still limited understanding on the non-asymptotic performance, especially under time-varying costs. Moreover, for computationally efficient algorithms, e.g. suboptimal MPC and inexact MPC  \cite{zeilinger2011real,wang2010fast,graichen2010stability,eichler2017optimal}, there is limited work on the optimality performance guarantee.

In online optimization, on the contrary, there are many papers on the non-asymptotic performance analysis, where the performance is usually measured by regret, e.g., static regrets\cite{hazan2016introduction,shalev2012online}, dynamic regrets\cite{jadbabaie2015online}, etc., but most work does not consider predictions and/or dynamical systems. Further, motivated by the applications with predictions, e.g. predictions of electricity prices in data center management problems \cite{lin2013dynamic,lin2012online}, there is a growing interest on  the effect of predictions on the online problems \cite{rakhlin2013online,chen2015online,lin2013dynamic,badiei2015online,lin2012online,chen2016using,li2018online}.
However, though some papers consider switching costs which can be viewed as a simple and special dynamical model \cite{goel2019online,li2018online}, there is a lack of study on the general dynamical systems and on how predictions affect the online problem with dynamical systems.

In this paper, we propose novel gradient-based online control algorithms, receding horizon gradient-based control (RHGC), and provide nonasymptotic optimality guarantees by dynamic regrets. RHGC can be based on many gradient methods, e.g.  vanilla gradient descent, Nesterov's accelerated gradient, triple momentum, etc., \cite{nesterov2013introductory,van2017fastest}. Due to the space limit, this paper only presents  receding horizon gradient descent (RHGD) and  receding horizon triple momentum (RHTM). For the theoretical  analysis, we assume strongly convex and smooth cost functions, whereas applying  RHGC does not require these conditions. Specifically, we show that  the regret bounds of RHGD and RHTM decay exponentially with the prediction window's size $W$, demonstrating that our algorithms efficiently utilize the prediction. Besides, our regret bounds  decrease when the system is more ``agile'' in the sense of a controllability index \cite{luenberger1967canonical}. Further, we provide a fundamental limit for any online control algorithms and show that the fundamental lower bound almost matches the regret upper bound of RHTM. This indicates that  RHTM achieves near-optimal performance at least in the worst case. We also provide some discussion on the classic linear quadratic tracking problems, a widely studied control problem in literature, to provide more insightful interpretations of our results. Finally, we numerically test our algorithms. In addition to linear systems, we also apply  RHGC to a nonlinear dynamical system: path tracking by a two-wheeled robot. Results show that  RHGC works effectively for nonlinear systems though RHGC is only presented and theoretical analyzed on LTI systems. 

Results in this paper are  built on a paper on online optimization with switching costs \cite{li2018online}. Compared with \cite{li2018online}, this paper studies online optimal control with \textit{general linear dynamics}, which includes  \cite{li2018online} as a special case; and studies how the system controllability index affects the regrets.


   There has been some recent work  on online  optimal  control problems with time-varying costs \cite{abbasi2014tracking,cohen2018online,goel2019online,agarwal2019online} and/or time-varying disturbances \cite{agarwal2019online},  but most papers  focus on the no-prediction cases. As we show later in this paper, these  algorithms can be used in our RHGC methods as initialization oracles. Moreover, our regret analysis shows that RHGC can reduce the regret of these no-prediction online algorithms by a factor exponentially decaying with the prediction window's size.

Finally, we would like to mention another related line of work:  learning-based control \cite{dean2017sample,dean2018regret,tu2017least,vamvoudakis2010online,ouyang2017learning}. In some sense, the results in this paper are orthogonal to that of the learning-based control, because the learning-based control usually considers a time-invariant environment but unknown dynamics,  and aims to learn system dynamics or optimal controllers by data; while this paper considers a time-varying scenario with known dynamics but changing objectives and studies decision making with limited predictions. It is an interesting future direction to combine the two lines of work for designing more applicable algorithms.

\nbf{Notations.}
Consider matrices $A$ and $B$, $A\geq B$ means $A-B$ is positive semidefinite and $[A,B]$ denotes a block matrix. The norm $\|\cdot\|$ refers to the $L_2$ norm for both vectors and matrices. Let $x^i$ denote the $i$th entry of the vector. Consider a set $\I=\{k_1, \dots, k_m \}$, then $x^\I=(x^{k_1},\dots, x^{k_m})^\top$, and $A(\I, :)$ denotes the $\I$ rows of matrix $A$ stacked together. Let $I_m$ be an identity matrix in $\R^{m \times m}$. 









\section{Problem formulation and preliminaries}







Consider a finite-horizon discrete-time optimal control problem with time-varying cost functions $f_t(x_t)+g_t(u_t)$ and a linear time-invariant (LTI) dynamical system:
{\small
\begin{equation}\label{equ: control problem}
\begin{aligned}
\min_{\mathbf x,\mathbf u}\quad& J(  \mathbf x,   \mathbf u)= \sum_{t=0}^{N-1} \left[ f_t(x_t)+   g_t(u_t) \right] + f_N(x_N)\\
\text{s.t. } \ & \ x_{t+1}=  A   x_t +   B   u_t, \quad t\geq 0
\end{aligned}
\end{equation}}
where  $x_t \in \R^n$, $u_t \in \R^m$, $\mathbf x=(x^\top_1, \dots, x^\top_N)^\top$, $\mathbf u=(u^\top_0,\dots, u^\top_{N-1})^\top$, $x_0$ is given,   $f_N(x_N)$ is the terminal cost.\footnote{The results in this paper can be extended to cost $c_t(x_t,u_t)$ with proper assumptions.}  To solve the optimal control problem  \eqref{equ: control problem},  all  cost functions from $t=0$ to $t=N$ are needed. However, at each time $t$, usually only a finite lookahead window of cost functions are available and the decision maker needs to make an online decision $u_t$ using the available information. 

In particular, we consider a simplified prediction model: at each time $t$, the decision maker obtains accurate predictions for the next $W$ time steps, $f_t,g_t, \dots, f_{t+W-1}, g_{t+W-1}$, but no further prediction beyond these $W$ steps, meaning that  $f_{t+W},g_{t+W},\dots$ can even be adversarially generated. Though this prediction model may be too optimistic in the short term and over pessimistic in the long term, this model i) captures a commonly observed phenomenon in predictions that short-term predictions are usually much more accurate than the long-term predictions; ii) allows researchers to derive insights for the role of predictions and possibly to extend to more complicated cases \cite{lin2012online,lin2013dynamic,lu2013online,borodin1992optimal}. 

The online optimal control problem is described as  follows: at each time step $t=0,1,\dots$, 
\begin{itemize}
		\item the agent observes state $x_t$ and receives prediction $f_t,\ g_t, \, \dots, \ f_{t+W-1},\ g_{t+W-1}$;
		\item the agent decides and implements a control $u_t$
		and  suffers the  cost $f_t(x_t)+   g_t(u_t) $;
		\item the system evolves to the next state $x_{t+1}=Ax_t + Bu_t$.\footnote{We assume known $A, B$, no process noises, state feedback, and leave relaxing assumptions as future work.}
\end{itemize}
	An online control algorithm, denoted as $\A$, can be defined as a mapping from the prediction information and the history information to the control action, denoted by $u_t(\mathcal{A})$:
	\begin{equation}\label{equ: online alg A for u}
	u_t(\A)=\A(   x_t(\A), \dots,   x_0(\A), \{  f_s,   g_s\}_{s=0}^{t+W-1}),\quad t\geq 0,
	\end{equation}
	where $x_t(\A)$ is the state generated by implementing $\A$ and $x_0(\A)=x_0$ is given.
	
This paper evaluates the performance of online control algorithms by comparing against the optimal control cost $J^*$ in hindsight, that is, $
J^* \coloneqq \min\{ J(  \mathbf x,  \mathbf u)\mid   x_{t+1}=Ax_t+Bu_t, \ \forall\, t \geq 0 \}$.

In this paper, the performance of an online algorithm $\A$    is measured by \footnote{The optimality gap depends on initial state $x_0$ and $\left\{f_t,g_t\right\}_{t=0}^{N}$, but we omit them for  simplicity of notation.}
\begin{equation}
\text{Regret}(\mathcal{A}):= J(\mathcal A)-J^* = J(\mathbf x(\A), \mathbf u(\A))-J^*,
\end{equation}
which is sometimes called as \textit{dynamic regret} \cite{jadbabaie2015online, mokhtari2016online} or \textit{competitive difference} \cite{andrew2013tale}. 
Another popular regret notion is the static regret, which compares the online performance with the optimal static controller/policy \cite{cohen2018online,abbasi2014tracking}. The benchmark in static regret is weaker than that in dynamic regret because the optimal controller may be far from being static, and it has been shown in literature that $o(N)$ static regret can be achieved even without predictions (i.e., $W=0$). Thus, we will focus on the dynamic regret analysis and study how predictions can improve the dynamic regret. 

\begin{example}[Linear quadratic (LQ) tracking]\label{example: LQT}
	Consider a discrete-time  tracking problem for a system
	$  x_{t+1}=  A   x_t +   B   u_t$.
	The goal is to minimize the quadratic tracking loss of a trajectory $\{\theta_t\}_{t=0}^N$ 
{\small\begin{align*}
	&  J(  \mathbf x,   \mathbf u)=\frac{1}{2} \sum_{t=0}^{N-1} \left[ (  x_t-  \theta_t)^\top   Q_t (  x_t-  \theta_t)+    u_t^\top   R_t u_t\right] + \frac{1}{2} (  x_{N}-  \theta_{N})^\top   Q_{N}(  x_{N}-  \theta_{N}).
	\end{align*} }
	In practice, it is usually difficult to know the complete trajectory $\{\theta_t\}_{t=0}^N$ a priori, what are revealed are usually the next few steps, making it an online control problem with predictions.


\end{example}

\textbf{Assumptions and  useful concepts.}
Firstly, we assume controllability, which is standard in control theory and roughly means that the system can be steered to any state by proper control inputs \cite{hespanha2018linear}.

\begin{assumption}\label{ass: controllable}
	The LTI system $x_{t+1}=Ax_t+Bu_t$ is controllable.
\end{assumption}
It is well-known that any controllable LTI system can be linearly transformed to a canonical form \cite{luenberger1967canonical} and the linear transformation can be computed efficiently a priori using $A$ and $B$, which can further be used to reformulate the cost functions $f_t, g_t$. Thus, without loss of generality, this paper only considers LTI systems in the canonical form, defined as follows.
\begin{definition}[Canonical form]\label{def: canonical form}
A system $x_{t+1}=Ax_t+Bu_t$ is said to be in the canonical form if 
{\small	\begin{align*}
	& A=\left[
	\begin{smallmatrix}
	0  					&	1 		& 	          & 0 &                      \\
	\vdots		 & \ddots & \ddots  & \\
	  		 & 		      & 	0		& 1		& \\
	*   & * & 		\cdots	 & * &  *& *& 	\dots		& 	*			&	& 		&\cdots	&& *\\
	&			  &				&  					& 0							 & 1        			&&0	& 		   	 	&	&			 						 & 					\\
	&			  &				&  						&  			\vdots	 & \ddots 				& 	\ddots	   	&&			&& 			&			  					\\
	&			  &				&  						&  				&   				         & 	           	   0	&			 1	&		&&		 & 	&	\\
	*  & * & 	\cdots		 &			*			&  *  & 		*				   & 	\cdots		& *  & 	\cdots	&*&\cdots&	&*\\
	\cdots     &			  &				&  						&&			 				&  				&		& \cdots    \\
	&           &&&&       &  &&& 0&1& \cdots&0\\
	&           &&&&       &  &&&\vdots &\ddots& \ddots&\\
	&           &&&&     &  &&& && 0&1\\
	*&   *        &\cdots&*&*&  *  &  \cdots &*&\cdots & *&*& \cdots& *
	\end{smallmatrix}
	\right], \quad B= \left[
	\begin{smallmatrix}
	0 			& 0 		&\dots \\
	\vdots  & \vdots & \vdots\\
	0 			& 		&\\
	1 			&  0 		& \cdots \\
	0 			&0& \\
	\vdots & \vdots& \cdots \\
	& & \\
	0 			& 1     & \cdots\\
	\cdots  &  &\cdots \\
	0 & \cdots &  \cdots\\
	\vdots & \ddots\\
	0 & 0 &  \cdots1
	\end{smallmatrix}
	\right],
	\end{align*}}
	where each * represents a (possibly) nonzero entry, and the rows of $B$ with $1$ are the same rows of $A$ with * and the indices of these rows are denoted as $\{k_1, \dots, k_m\}\eqqcolon \I$. 
	Moreover, let $p_i = k_i -k_{i-1}$ for  $ 1 \leq i \leq m$, where $k_0=0$. The \textup{controllability index} of a canonical-form $(A,B)$ is defined as $$	p = \max\{p_1, \dots, p_m \}.$$
 \end{definition}
Next, we introduce assumptions on the cost functions and their optimal solutions. 
\begin{assumption}\label{ass: general Qt, Rt factors}
Assume $f_t$ is $\mu_f$ strongly convex and $l_f$ Lipschitz smooth for  $0\leq t \leq N$, and $g_t$ is { convex} and $l_g$ Lipschitz smooth for $0 \leq t \leq N-1$ for some $\mu_f,   l_f, l_g>0$. 
	
	\end{assumption}
	
	\begin{assumption}\label{ass: theta xi bdd}Assume the minimizers to $f_t, g_t$, denoted as $\theta_t = \argmin_x f_t(x),\ \xi_t =  \argmin_u g_t(u)$, are uniformly bounded,  i.e. there exist $\bar \theta, \bar \xi$ such that $\| \theta_t \|\leq \bar \theta$, $\| \xi_t \| \leq \bar \xi, \ \forall\, t$.

\end{assumption}

These assumptions are commonly adopted in convex analysis. The uniform bounds  rule out extreme cases. Notice  that the LQ tracking problem in Example 1 satisfies  Assumption \ref{ass: general Qt, Rt factors} and \ref{ass: theta xi bdd} if $Q_t, R_t$ are positive definite with uniform bounds on the eigenvalues and if $\theta_t$ are uniformly bounded for all $t$. 


\section{Online control algorithms: receding horizon gradient-based control}
\label{sec: RHTM}


This section introduces our online control algorithms, receding horizon gradient-based control (RHGC). The design is by first converting the online control problem to an equivalent online optimization problem with \textit{finite temporal-coupling} costs, then designing gradient-based online optimization algorithms by utilizing this finite temporal-coupling property.

\subsection{Problem transformation}\label{sec: problem transform}

Firstly, we notice that the offline optimal control problem \eqref{equ: control problem} can be viewed as an optimization with equality constraints over $\mbf x$ and $\mbf u$. The individual stage cost $f_t(x_t)+g_t(u_t)$ only depends on the current $x_t$ and $u_t$ but the equality constraints couple $x_t$, $u_t$ with $x_{t+1}$ for each $t$. In the following, we will rewrite \eqref{equ: control problem} in an equivalent form of an \textit{unconstrained} optimization problem on some entries of $x_t$ for all $t$, but the new stage cost at each time $t$ will depend on these new entries across a few nearby time steps. We will harness this structure to design our online algorithm.  

In particular, the entries of $x_t$ adopted in the reformulation are: $x_t^{k_1},\dots, x_t^{k_m}$, where $\I=\{k_1, \dots, k_m\}$ is defined in Definition \ref{def: canonical form}. For ease of notation, we define 
\begin{equation} \label{equ: zt_def}
    z_t\coloneqq (x_t^{k_1},\dots, x_t^{k_m})^\top, \quad t\geq 0
\end{equation} 
and write $z_t^j=x_t^{k_j}$ where $j=1,\dots,m$. Let $\mbf z\coloneqq (z^\top_{1},\dots, \dots, z^\top_N)^\top$. By the  canonical-form equality constraint $x_{t}=Ax_{t-1}+Bu_{t-1}$, we have $x_{t}^i=x_{t-1}^{i+1}$ for $i\not \in \I$, so $x_t$ can be represented by $z_{t-p+1},\dots, z_t$ in the following way:
\begin{equation}\label{equ: xt's representation by z}
    x_t=(\underbrace{z_{t-p_1+1}^1, \dots, z_{t}^1}_{p_1}, \underbrace{z_{t-p_2+1}^2, \dots, z_t^2}_{p_2}, \dots, \underbrace{z_{t-p_m+1}^m, \dots, z_t^m}_{p_m})^\top, \quad t\geq0,
\end{equation}
where $z_t$ for $t\leq 0$ is determined by $x_0$ in a way to let~\eqref{equ: xt's representation by z} hold for $t=0$. For  ease of  exposition and without loss of generality,  we consider $x_0=0$ in this paper; then we have $z_t=0$ for $t\leq 0$.
Similarly, $u_t$ can be determined by $z_{t-p+1}, \dots, z_t, z_{t+1}$ by 
\begin{equation}\label{equ: ut= zt+1- mr Axt}
    u_t = z_{t+1}- A(\I,:) x_t= z_{t+1}-A(\I,:)(z_{t-p_1+1}^1, \dots, z_{t}^1,  \dots, z_{t-p_m+1}^m, \dots, z_t^m)^\top, \quad t\geq 0
\end{equation} 
where $A(\I,:)$ consists of $k_1, \dots, k_m$ rows of $A$. 

It is straightforward to verify that equations (\ref{equ: zt_def}, \ref{equ: xt's representation by z}, \ref{equ: ut= zt+1- mr Axt}) describe a bijective transformation between $\{(\mbf x, \mbf u)\mid x_{t+1}=Ax_t +Bu_t\}$  and $\mbf z \in \R^{mN}$, since the LTI constraint $x_{t+1}=Ax_t +Bu_t$ is naturally embedded in the relation (\ref{equ: xt's representation by z}, \ref{equ: ut= zt+1- mr Axt}). Therefore, based on the transformation, an optimization problem with respect to $\mbf z \in \R^{m N}$ can be designed to be equivalent with \eqref{equ: control problem}. Notice that the resulting optimization problem  has no constraint on $\mbf z$. Moreover, the cost functions on $\mbf z$ can be obtained by substituting  (\ref{equ: xt's representation by z}, \ref{equ: ut= zt+1- mr Axt}) into $f_t(x_t)$ and $g_t(u_t)$, i.e. $\tilde{f}_t(z_{t-p+1}, \dots, z_t)\coloneqq f_t(x_t)$ and $\tilde{g}_t(z_{t-p+1}, \dots, z_t, z_{t+1})\coloneqq g_t(u_t)$. 
Correspondingly, the objective function of the equivalent optimization with respect to $\mbf z$ is
\begin{align}\label{equ: def C(z)}
& C(\mathbf z)\coloneqq \sum_{t=0}^{N} \tilde f_t(z_{t-p+1},\dots, z_t)+\sum_{t=0}^{N-1} \tilde g_t(z_{t-p+1},\dots,z_{t+1})
\end{align}
 $C(\mbf z)$ has many nice properties, some of which are formally stated  below. 
 
\begin{lemma}\label{lem: C(z) properties}The function $C(\mbf z)$ has the following properties:
\vspace{-4pt}
\begin{enumerate}
\itemsep-1pt
    \item[i)] $C(\mbf z)$ is $\mu_c=\mu_f$ strongly convex and $l_c$ smooth for  $l_c= pl_f + (p+1)l_g \|[I_m,-A(\I,:)]\|^2$.
    \item[ii)] For any $ (\mbf x,\mbf u)$ s.t. {$x_{t+1}= Ax_t+B u_t$}, $C(\mbf z)=J(\mbf x,\mbf u)$ where $\mbf z$ is defined in (\ref{equ: zt_def}). Conversely,   $ \forall\,\mbf z$, the $(\mbf x, \mbf u)$ determined by (\ref{equ: xt's representation by z},\ref{equ: ut= zt+1- mr Axt}) satisfies   $x_{t+1}= A x_t+ B u_t$ and $J(\mbf x,\mbf u)=C(\mbf z)$;
    \item[iii)] Each stage cost $\tilde{f}_t+\tilde{g}_t$ in (\ref{equ: def C(z)}) only depends on $z_{t-p+1},\ldots,z_{t+1}$. 
\end{enumerate}
\end{lemma}

Property ii) implies that any online algorithm for deciding $\mbf z$ can be translated to an online algorithm for $\mbf x$ and $\mbf u$ by (\ref{equ: xt's representation by z}, \ref{equ: ut= zt+1- mr Axt}) with the same costs.
Property iii) highlights one nice property, finite temporal-coupling, of $C(\mbf z)$, which serves as a foundation for our online algorithm design.  



\begin{example}\label{example: p=2 n=2} For illustration,  
consider the following dynamical system with $n=2,\ m=1$:
{\small	\begin{align}\label{equ: example n=2}
 \left[
	\begin{array}{c}
	x_{t+1}^1\\
    x_{t+1}^2
	\end{array}
	\right]= \left[
	\begin{array}{cc}
	0 & 1\\
	a_1 & a_2
	\end{array}
	\right] \left[
	\begin{array}{c}
	x_{t}^1\\
    x_{t}^2
	\end{array}
	\right]+ 
	\left[
	\begin{array}{c}
	0\\
	1
	\end{array}
	\right]u_{t}
	\end{align}}
	Here, $k_1=2$, $\mathcal{I}=\left\{2\right\}$, $A(\mathcal{I},:)=(a_1, a_2)$, and $z_t=x_t^2$.  By \eqref{equ: example n=2},
 $x_{t}^1= x_{t-1}^2$ and $x_t=(z_{t-1},z_t)^\top$. Similarly, $u_t=x^2_{t+1}- A(\mathcal{I},:) x_t=z_{t+1}- A(\mathcal{I},:)(z_{t-1},z_t)^\top$. Hence,  $\tilde{f}_t(z_{t-1}, z_t)= f_t(x_t)=f_t((z_{t-1},z_t)^\top)$, $\tilde {g}_t(z_{t-1},z_t, z_{t+1})= g_t(u_t)=g_t(z_{t+1}- A(\mathcal{I},:) (z_{t-1},z_t)^\top )$.  
\end{example}
\begin{remark} This paper considers a  reparameterization method with respect to states $\mbf x$ via the canonical form, and it might be interesting to compare it with the  more direct reparameterization with respect to control inputs $\mbf u$. The control-based reparameterization has been discussed in literature \cite{richter2011computational}. It has been observed in \cite{richter2011computational} that when $A$ is not stable, the condition number of the cost function derived from the control-based reparameterization goes to infinity as $W\to +\infty$, which may result in computation issues when $W$ is large. However, the state-based reparameterization considered in this paper can guarantee bounded condition number for all $W$ even for unstable $A$, as shown in Lemma \ref{lem: C(z) properties}. This is one major advantage of the state-based reparameterization method considered in this paper.
\end{remark}

\subsection{Online algorithm design: RHGC}\label{sec: rhgm}
This section introduces our RHGC  based on the reformulation (\ref{equ: def C(z)}) and inspired by  \cite{li2018online}. As mentioned earlier, any online algorithm for $z_t$ can be translated to  an online algorithm for $x_t,u_t$. Hence, we will focus on designing an online algorithm for $z_t$ in the following. By the finite temporal-coupling property of $C(\mbf z)$, the partial gradient of the \textit{total cost} $C(\mbf z)$ only depends on the finite neighboring stage costs $\{\tilde{f}_\tau, \tilde{g}_{\tau}\}_{\tau=t}^{t+p-1}$ and finite neighboring stage variables $(z_{t-p},\ldots,z_{t+p})=:z_{t-p:t+p}$. 
{\small\begin{align*}
&\frac{\partial C}{\partial z_t}(\mbf z)= \sum_{{\tau}=t}^{t+p-1} \frac{\partial \tilde f_{\tau}}{\partial z_t}(z_{{\tau}-p+1},\dots, z_{\tau}) + \sum_{{\tau}=t-1}^{t+p-1} \frac{\partial \tilde g_{\tau}}{\partial z_t}(z_{{\tau}-p+1},\dots, z_{{\tau}+1})
\end{align*}}
Without causing any confusion, we use $\frac{\partial C}{\partial z_{t}}(z_{t-p:t+p})$ to denote $\frac{\partial C}{\partial z_t}(\mbf z)$ for highlighting the local dependence. Thanks to the local dependence, despite  the fact that not all the future costs are available, it is still possible to compute the partial gradient of the total cost by using only a finite lookahead window of the cost functions. This observation motivates the design of our receding horizon gradient-based control (RHGC) methods, which are the online implementation of gradient methods, such as  vanilla gradient descent, Nesterov's accelerated gradient, triple momentum, etc., \cite{nesterov2013introductory,van2017fastest}. 


\begin{algorithm}
\caption{\black{Receding Horizon Gradient Descent (RHGD)}}
	\label{alg:RHGD-LQT}
	\begin{algorithmic}[1]
		\STATE \textbf{inputs:} Canonical form $(A,B)$, $W\geq 1$, $K= \floor{\frac{W-1}{p}}$, stepsize $\gamma_g$, initialization oracle $\varphi$.
		\FOR{$t=1-W:N-1$}
		\STATE  \textit{Step 1:} initialize $z_{t+W}(0)$ by  oracle $\varphi$.
		
		\FOR{$j=1,\dots, K$}
		\STATE \textit{Step 2:} update $z_{t+W-jp}(j)$ by gradient descent
$z_{t+W-jp}(j)=  z_{t+W-jp}(j-1)-\gamma_g \frac{\partial C}{\partial z_{t+W-jp}}(z_{t+W-(j+1)p:t+W-(j-1)p}(j-1))
$.
		\ENDFOR
		\STATE \textit{Step 3:} compute $u_t$ by $z_{t+1}(K)$ and the observed state $x_t$:
		$u_t = z_{t+1}(K) - A(\I, :) x_t$
		\ENDFOR
	\end{algorithmic}
\end{algorithm}

\black{Firstly, we illustrate the main idea of RHGC by receding horizon gradient descent (RHGD) based on vanilla gradient descent. In RHGD (Algorithm \ref{alg:RHGD-LQT}), index $j$ refers to the iteration number of the corresponding gradient update of $C(\mbf z)$. There are two major steps to decide $z_t$. Step 1 is initializing the decision variables $\mbf z(0)$. Here, we do not restrict the initialization algorithm $\varphi$ and allow any oracle/online algorithm without using lookahead information, i.e.  $z_{t+W}(0)$ is selected based only on the information up to $t+W-1$: $z_{t+W}(0)=\varphi(\{\tilde f_s, \tilde g_s \}_{s=0}^{t+W-1})$.
One example of $\varphi$ will be provided in Section 4. Step 2 is using the $W$-lookahead costs to conduct gradient updates. Notice that the gradient update from $z_{\tau}(j-1)$ to $z_{\tau}(j)$ is implemented in a backward order of $\tau$, i.e. from $\tau=t+W$ to $\tau=t$. Moreover, since the partial gradient  $\frac{\partial C}{\partial z_t}$ requires the local decision variables $z_{t-p:t+p}$, given $W$-lookahead information,  RHGD can only conduct $K=\floor{\frac{W-1}{p}}$ iterations of gradient descent for the total cost  $C(\mbf z)$. For more  discussion,  we  refer the reader to \cite{li2018online} for the $p=1$ case. }

In addition to RHGD, RHGC can also incorporate accelerated gradient methods in the same way, such as Nesterov's accelerated gradient and triple momentum. For the space limit, we only formally present receding horizon triple momentum (RHTM) in Algorithm \ref{alg:RHTM-LQT} based on triple momentum \cite{van2017fastest}. 
RHTM also consists of two major steps when determining $z_t$: initialization and gradient updates based on the lookahead window. The two major differences from RHGD are that the decision variables in RHTM include not only $ z(j)$ but also auxiliary variables $ \omega(j)$ and $  y(j)$, which are adopted in triple momentum to accelerate the convergence, and  that the gradient update  is by triple momentum instead of gradient descent. Nevertheless, RHTM can also conduct $K=\floor{\frac{W-1}{p}}$ iterations of triple momentum for  $C(\mbf z)$ since the triple momentum update requires the same neighboring cost functions.

Though it appears that RHTM does not fully exploit the lookahead information since only a few gradient updates are used, in Section~\ref{sec:LQT}, we show that RHTM  achieves near-optimal performance with respect to $W$, which means that RHTM successfully extracts and utilizes the prediction information. 


Finally, we briefly introduce MPC\cite{rawlings2012postface} and suboptimal MPC\cite{zeilinger2011real}, and compare them with our algorithms. MPC tries to solve a $W$-stage optimization at each  $t$ and implements the first control input. Suboptimal MPC, as a variant of MPC aiming at reducing computation, conducts an optimization method only for a few iterations without solving the optimization completely. Our algorithm's computation time is similar to that of suboptimal MPC with a few gradient iterations. 
However, the major difference between our algorithm and suboptimal MPC is that \black{suboptimal MPC conducts gradient updates for a truncated $W$-stage optimal control problem based on $W$-lookahead information, while our algorithm is able to conduct  gradient updates for the complete $N$-stage optimal control problem based on the same $W$-lookahead information by utilizing the reformulation  (\ref{equ: zt_def}, \ref{equ: xt's representation by z}, \ref{equ: ut= zt+1- mr Axt}, \ref{equ: def C(z)}).}

\section{Regret upper bounds}

Because our RHTM (RHGD) is designed to exactly implement the triple momentum (gradient descent) of $C(\mbf z)$ for $K$ iterations, it is straightforward to have the following regret guarantees  that connect  the regrets of RHTM and RHGD with the regret of the initialization oracle $\varphi$,

\begin{algorithm}
\caption{Receding Horizon Triple Momentum (RHTM)}
	\label{alg:RHTM-LQT}
	\begin{algorithmic}
		\STATE \textbf{inputs:} Canonical form $(A,B)$, $W\geq 1$, $K= \floor{\frac{W-1}{p}}$, stepsizes $\gamma_c, \gamma_z, \gamma_{\omega}, \gamma_y>0$,  oracle $\varphi$.
		\FOR{$t=1-W:N-1$}
		\STATE  \textit{Step 1:} initialize $z_{t+W}(0)$ by  oracle $\varphi$, then let $\omega_{t+W}(-1),\omega_{t+W}(0), y_{t+W}(0)$ be $z_{t+W}(0)$
		
		\FOR{$j=1,\dots, K$}
		\STATE \textit{Step 2:} update $\omega_{t+W-jp}(j), y_{t+W-jp}(j),z_{t+W-jp}(j)$ by triple momentum.
	\begin{align*}
		&\omega_{t+W-jp}(j)= (1+\gamma_{\omega}) \omega_{t+W-jp}(j-1)-\gamma_{\omega} \omega_{t+W-jp}(j-2)\\
		& \quad \qquad \qquad \qquad 
		-\gamma_c  \frac{\partial C}{\partial y_{t+W-jp}}(y_{t+W-(j+1)p:t+W-(j-1)p}(j-1))\\
		& y_{t+W-jp}(j) = (1+\gamma_y) \omega_{t+W-jp}(j) -\gamma_y \omega_{t+W-jp}(j-1)\\
		& z_{t+W-jp}(j)=(1+\gamma_z)\omega_{t+W-jp}(j) -\gamma_z \omega_{t+W-jp}(j-1)
		\end{align*}
		\ENDFOR
		\STATE \textit{Step 3:} compute $u_t$ by $z_{t+1}(K)$ and the observed state $x_t$:
		$u_t = z_{t+1}(K) - A(\I, :) x_t$
		\ENDFOR
	\end{algorithmic}
\end{algorithm}

\begin{theorem}\label{thm: RHGM regret bdd} 
	Consider $W\geq 1$ and stepsizes $\gamma_g=\frac{1}{l_c}$, $\gamma_c = \frac{1+\phi}{l_c}$, $ \gamma_{\omega} = \frac{\phi^2}{2-\phi}$, $ \gamma_y = \frac{\phi^2}{(1+\phi)(2-\phi)}$, $ \gamma_z = \frac{\phi^2}{1-\phi^2}$,  $\phi= 1-1/\sqrt{\zeta}$, 
	and let $\zeta=l_c/\mu_c$ denote $C(\mbf z)$'s condition number. For any oracle $\varphi$, 
	\begin{align*}
	&\black{\textup{Regret}(RHGD) \leq \zeta \left(\frac{{\zeta}-1}{ \zeta}\right)^{K}\textup{Regret}(\varphi)}, \quad \textup{Regret}(RHTM) \leq \zeta^2 \left(\frac{\sqrt{\zeta}-1}{\sqrt \zeta}\right)^{2K}\textup{Regret}(\varphi)
	\end{align*}
	where $K= \floor{\frac{W-1}{p}}$, $\textup{Regret}(\varphi)$ is the regret of  the initial controller: $u_t(0)=z_{t+1}(0)-A(\I, :)x_t(0)$.
\end{theorem}

 Theorem~\ref{thm: RHGM regret bdd} suggests that  for any online algorithm $\varphi$ without predictions, RHGD and RHTM  can use predictions to lower the regret by a factor of $\zeta (\frac{\zeta-1}{\zeta})^K$ and $\zeta^2(\frac{\sqrt{\zeta}-1}{\sqrt \zeta})^{2K}$ respectively via additional $K=\floor{\frac{W-1}{p}}$ gradient updates.
Moreover, the factors decay exponentially with $K=\floor{\frac{W-1}{p}}$, and $K$ almost linearly increases with $W$. This indicates that  RHGD and RHTM   improve the performance exponentially fast with an increase in the prediction window $W$ for any initialization method. In addition, $K=\floor{\frac{W-1}{p}}$ decreases with  $p$, implying that the regrets increase with the controllability index $p$ (Definition~\ref{def: canonical form}). This is intuitive because $p$ roughly indicates how fast the controller can influence the system state effectively: the larger the $p$ is, the longer it takes. 
To see this, consider Example \ref{example: p=2 n=2}. Since $u_{t-1}$ does not directly affect $x_t^1$, it takes at least $p=2$ steps to change $x_t^1$ to a desirable value. Finally, RHTM's regret decays faster than RHGD's, which is intuitive because triple momentum converges faster than gradient descent. Thus, we will focus on  RHTM in the following.


\nbf{An initialization method: follow the optimal steady state (FOSS).}
To complete the regret analysis for RHTM, we provide a simple initialization method, FOSS, and its dynamic regret bound.
As
mentioned before, any online control algorithm without predictions, e.g.  \cite{cohen2018online,abbasi2014tracking}, can be applied as an initialization oracle $\varphi$. However, most literature  study  static regrets  rather than dynamic regrets.

\begin{definition}[Follow the optimal steady state (FOSS)]\label{def: follow the optimal steady state}
The optimal steady state for stage cost $f(x)+g(u)$ refers to $(x^e,u^e)\coloneqq\argmin_{x=Ax+Bu}(f(x)+g(u))$.

 Follow the optimal steady state algorithm (FOSS) first solves the optimal steady state $(x_t^e, u_t^e)$ for cost  $f_t(x)+g_t(u)$, then determines $z_{t+1}$ by $x_t^e$, i.e. 
 $z_{t+1}=(x_t^{e,k_1}, \dots, x_t^{e,k_m})^\top$ at each $t+1$.
\end{definition}
FOSS is motivated by the fact that the optimal steady state cost is the optimal infinite-horizon average cost for LTI systems with time-invariant cost functions \cite{angeli2009receding}, so FOSS should yield acceptable performance at least for slowly changing cost functions. Nevertheless, we admit that  FOSS is proposed mainly for analytical purposes and other online algorithms may outperform FOSS in various perspectives.
 The following is a regret bound for FOSS, relying on the solution to  Bellman equations.
\begin{definition}[Solution to the Bellman equations \cite{puterman2014markov}]
Consider optimal control problem: $\min \lim_{N\to +\infty}  \frac{1}{N}\sum_{t=0}^{N-1} (f(x_t)+g(u_t))$ where $x_{t+1}=Ax_t + Bu_t$.
    Let $\lambda^e$ be the optimal steady state cost $f(x^e)+g(u^e)$, which is also the optimal infinite-horizon average cost \cite{angeli2009receding}. The Bellman equations for the    problem is $h^e(x)+\lambda^e = \min_u (f(x)+g(u)+h^e(Ax+Bu))$.
    The solution to the Bellman equations, denoted by $h^e(x)$, is sometimes called as a \textit{bias function} \cite{puterman2014markov}. To ensure the uniqueness of the solution,  some extra conditions, e.g. $h^e(0)=0$, are usually  imposed.
\end{definition}




\begin{theorem}[Regret bound of FOSS]\label{thm: initialization}
Let $(x_t^e,u_t^e)$ and $h_t^e(x)$ denote the optimal steady state and the bias function with respect to cost $f_t(x)+g_t(u)$ respectively for $0\leq t \leq N-1$. Suppose $h_t^e(x)$ exists for $0\leq t \leq N-1$,\footnote{$h_t^e$ may not be unique, so extra conditions  can be imposed on $h_t^e$ for more interesting    regret bounds.} then the regret of FOSS can be bounded by
\begin{align*}
	& \textup{Regret}(\text{FOSS})= O\left(\sum_{t=0}^{N}(\| x_{t-1}^e- x_t^e \| +h^e_{t-1}(x_{t}^*)-h^e_{t}(x_{t}^*))\right),
	\end{align*}
	where
	  $\{x_t^*\}_{t=0}^{N}$ denotes the optimal state trajectory for \eqref{equ: control problem}, $x_{-1}^e=x_0^*=x_0=0$, $h^e_{-1}(x)=0, h^e_N(x)=f_N(x)$, $x_N^e=\theta_N$.
	  Consequently, by Theorem \ref{thm: RHGM regret bdd}, the regret bound of RHTM with initialization FOSS is
$ \textup{Regret}(\text{RHTM})= O\left(\zeta^2(\frac{\sqrt{\zeta}-1}{\sqrt \zeta})^{2K}\sum_{t=0}^{N}(\| x_{t-1}^e- x_t^e \| +h^e_{t-1}(x_{t}^*)-h^e_{t}(x_{t}^*))\right).
	$
\end{theorem}

Theorem \ref{thm: initialization} bounds the regret by the variation of the optimal steady states $x_t^e$ and the bias functions $h_t^e$. 
If $f_t$ and $ g_t$ do not change, $x_t^e$ and  $h_t^e$ do not change, yielding a small $O(1)$ regret, i.e. $O(\|x_0^e\|+h_0^e(x_0))$,  matching our intuition. Though Theorem \ref{thm: initialization} requires  $h_t^e$ exists, the existence is guaranteed for many control problems, e.g. LQ tracking and control problems with turnpike properties \cite{damm2014exponential,grune2018economic}. 

\section{Linear quadratic tracking: regret upper bounds and a fundamental limit}\label{sec:LQT}

To provide more intuitive meaning for our regret analysis in Theorem~\ref{thm: RHGM regret bdd} and Theorem~\ref{thm: initialization}, we apply RHTM to the LQ tracking problem in Example~\ref{example: LQT}. Results for the time varying $Q_t, R_t, \theta_t$ are provided in  Appendix E; whereas here we focus on a special case which gives clean expressions for regret bounds: both an upper bound for RHTM with initialization FOSS and a lower bound for any online algorithm. Further, we show that the lower bound and the upper bound almost match each other, implying that our online algorithm RHTM uses the predictions in a nearly optimal way even though it only conducts a few gradient updates at each time step . 

The special case of LQ tracking problems is in the following form,
	\begin{equation}\label{equ: LQT Q R not changing}
	    \frac{1}{2} \sum_{t=0}^{N-1} \left[ (  x_t-  \theta_t)^\top   Q (  x_t-  \theta_t)+    u_t^\top   R u_t\right] + \frac{1}{2}   x_{N}^\top   P^e  x_{N},
	\end{equation}
	where $Q>0$, $R>0$, and $P^e$ is the solution to the algebraic Riccati equation with respect to $Q, R$ \cite{bertsekas3rddynamic}. Basically, in this special case, $Q_t=Q$, $R_t=R$ for $0\leq t \leq N-1$, $Q_N=P^e$, $\theta_N=0$, and only $\theta_t$ changes for $t=0, 1, \dots, N-1$. The  LQ  tracking problem \eqref{equ: LQT Q R not changing} aims to follow a time-varying trajectory $\{\theta_t\}$ with constant weights on the tracking cost and the control cost. 

\nbf{Regret upper bound.} Firstly, based on Theorem~\ref{thm: RHGM regret bdd} and Theorem~\ref{thm: initialization}, we have the following bound.
	\begin{corollary}\label{cor: LQT Q, R not change}
	Under the stepsizes in Theorem \ref{thm: RHGM regret bdd},  RHTM with FOSS as the initialization rule satisfies
	\begin{align*}
	& \textup{Regret}(RHTM) = O\left(\zeta^2 (\frac{\sqrt{\zeta}-1}{\sqrt \zeta})^{2K}\sum_{t=0}^{N}  \|   \theta_t-\theta_{t-1}\| \right)
	\end{align*}
	where  $K= \floor{(W-1)/p}$,   $\zeta$ is the condition number of the corresponding $C(\mbf z)$,  $\theta_{-1}=0$.	
\end{corollary}
This corollary shows that the regret can be bounded by the total variation of $\theta_t$ for constant $Q, R$. 

\nbf{Fundamental limit.} For any online algorithm, we have the following lower bound.
\begin{theorem}[Lower Bound]\label{thm: lower bdd}
Consider \black{$1 \leq W\leq N/3$}, any condition number $\zeta >1$, any variation budget \black{$4 \bar \theta \leq L_N \leq  (2N+1) \bar \theta$}, and any controllability index $p\geq 1$. For any online algorithm $\A$, there exists an LQ tracking problem in form \eqref{equ: LQT Q R not changing} where i) the canonical-form system $(A, B)$ has controllability index $p$, ii) the sequence $\{\theta_t\}$ satisfies the variation budget $\sum_{t=0}^N \|   \theta_t-\theta_{t-1}\|\leq L_N$, and iii) the corresponding $C(\mbf z)$ has condition number $\zeta$, such that the following  lower bound holds
	\begin{align}
	\textup{Regret}(\mathcal{A}) = \Omega\left(  (\frac{\sqrt \zeta -1}{\sqrt \zeta +1})^{2 K}L_N \right)= \Omega\left(  (\frac{\sqrt \zeta -1}{\sqrt \zeta +1})^{2 K}\sum_{t=0}^N \|   \theta_t-\theta_{t-1}\|\right)
	\end{align}
	where   $K = \floor{(W-1)/p}$ and  $\theta_{-1}=0$.
\end{theorem}
Firstly, the lower bound in Theorem \ref{thm: lower bdd} almost matches the upper bound in Corollary \ref{cor: LQT Q, R not change}, especially when $\zeta$ is large, demonstrating that RHTM utilizes the predictions  in a near-optimal way. The major conditions in Theorem \ref{thm: lower bdd} require that the prediction window is short compared with the horizon: $W\leq N/3$, and the variation of the cost functions should not be too small: $L_N\geq 4\bar \theta$, otherwise the online control problem is too easy and the regret can be very small. Moreover, the small gap between the regret bounds is conjectured to be nontrivial, because this gap coincides with the long lasting  gap in the convergence rate of the first-order algorithms for strongly convex and smooth optimization. In particular, the lower bound in Theorem \ref{thm: lower bdd} matches the fundamental convergence limit  in \cite{nesterov2013introductory}, and the upper bound is by triple momentum's convergence rate, which is the best one to our knowledge. 

\section{Numerical experiments}


\begin{figure}[htp]
    \centering
    \begin{minipage}[b]{.4\textwidth}
    \centering
     \includegraphics[width=1\textwidth]{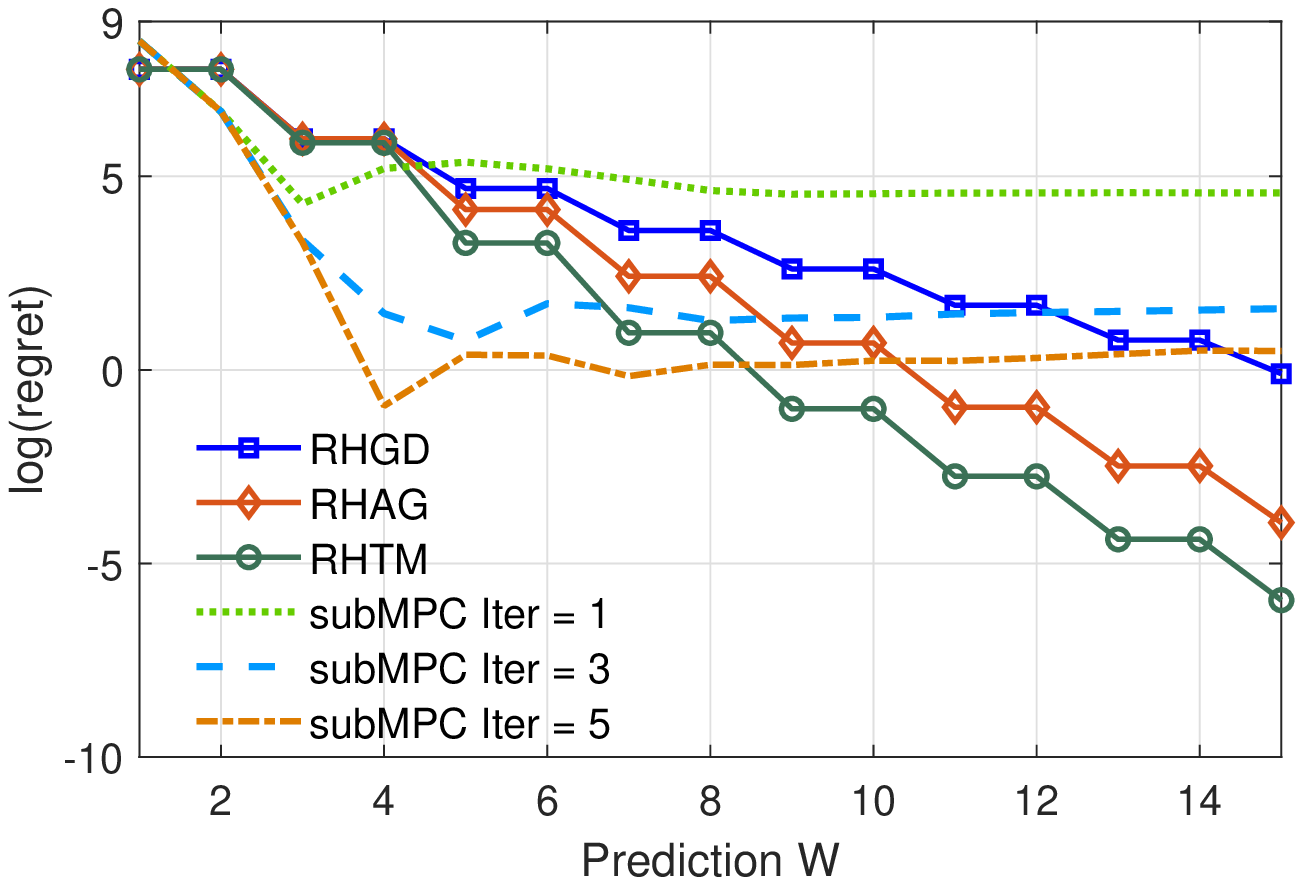}
    \begin{minipage}[t]{\textwidth}
    \caption{Regret for LQ tracking.}
    \label{fig:LQT}
    \end{minipage}
    \end{minipage}%
    \hfill
    \begin{minipage}[b]{.6\textwidth}
    \centering
   \includegraphics[width=.5\textwidth]{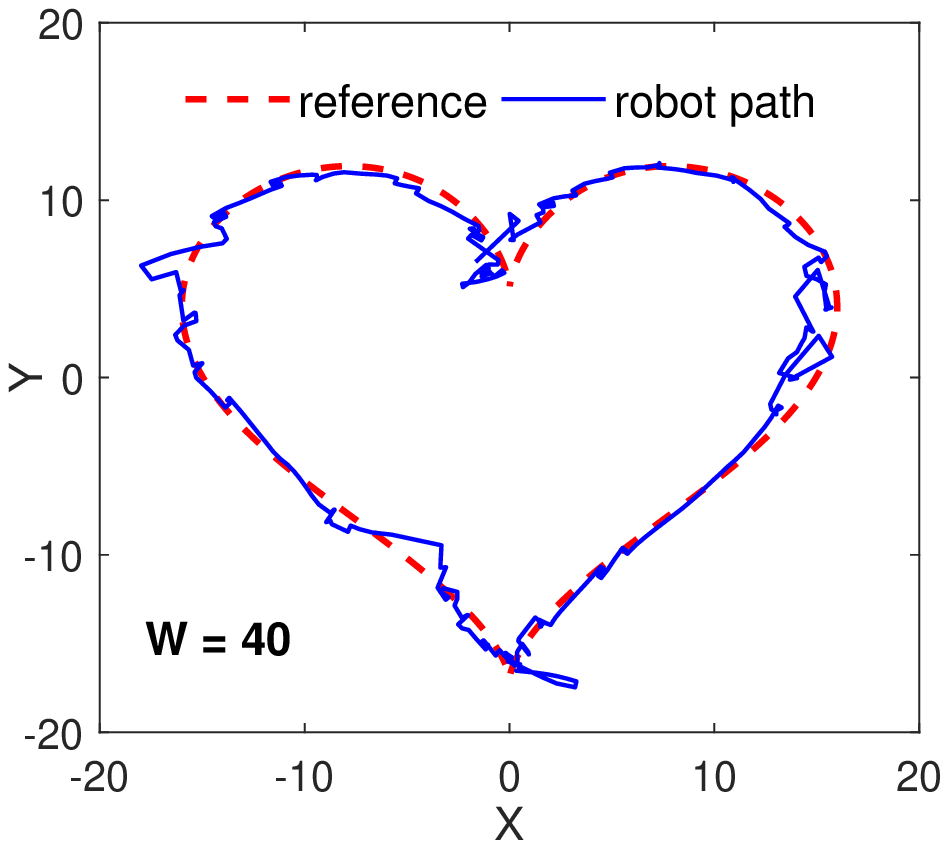}\includegraphics[width=.5\textwidth]{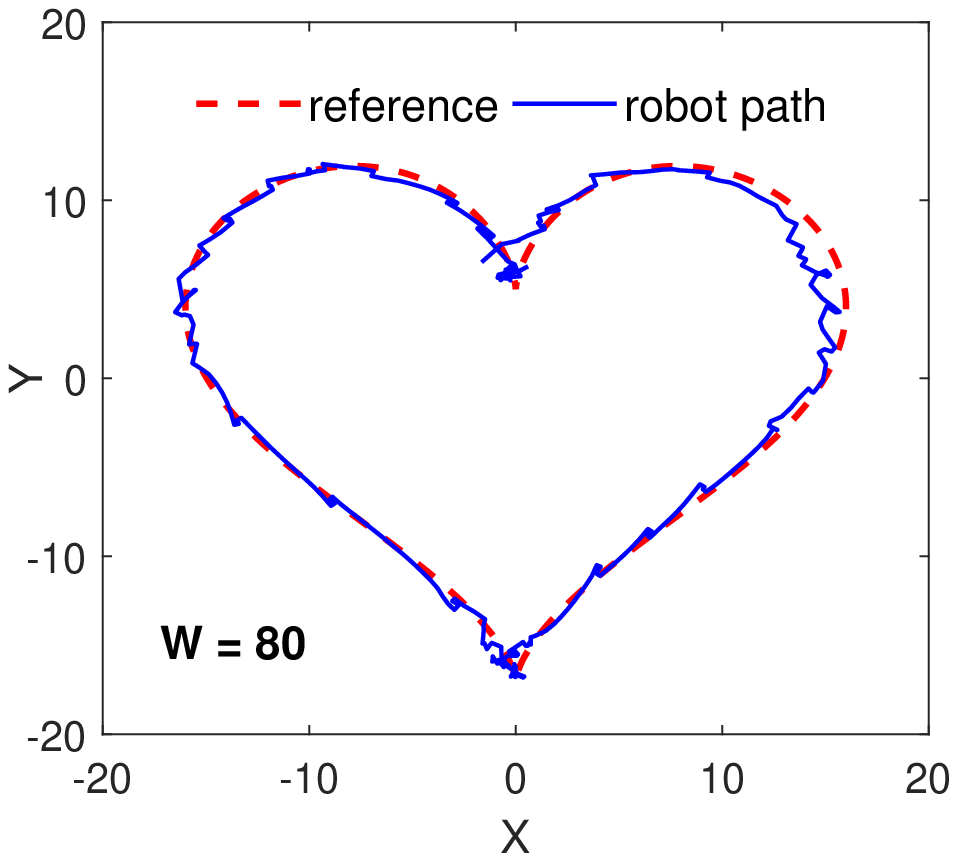}
    \begin{minipage}[t]{\textwidth}
    \caption{Two-wheel robot tracking with nonlinear dynamics.}
    \label{fig: robot}
    \end{minipage}
    \end{minipage}
    \vspace{-9pt}
\end{figure}

\textbf{LQ tracking problem in Example~\ref{example: LQT}.} The system considered here has $n=2$, $m=1$, and $p=2$. More details  of the experiment settings are provided in Appendix H. 
 We compare RHGC with a suboptimal MPC algorithm, fast gradient MPC (subMPC) \cite{zeilinger2011real}. Roughly speaking, subMPC solves the $W$-stage truncated optimal control from $t$ to $t+W-1$ by Nesterov's accelerated gradient \cite{nesterov2013introductory}, and one iteration of Nesterov's accelerated gradient requires $2W$ gradient evaluations of stage cost function since $W$ stages are considered and each stage has two costs $f_t$ and $g_t$.  This implies that, in terms of the number of gradient evaluations, subMPC with one iteration corresponds to our RHTM because RHTM also requires roughly $2W$ gradient evaluations per stage. Therefore, Figure \ref{fig:LQT} compares our RHGC algorithms with subMPC with one iteration. Figure \ref{fig:LQT} also plots subMPC with 3 and 5 iterations for more insights. Besides, Figure \ref{fig:LQT} plots not only RHGD and RHTM, but also RHAG, which is based on Nesterov's accelerated gradient.  
Figure \ref{fig:LQT} shows that all our  algorithms achieve exponential decaying regrets with respect to $W$, and the regrets are piecewise constant, matching Theorem \ref{thm: RHGM regret bdd}.  Further, it is observed that RHTM and RHAG perform better than RHGD, which is intuitive because triple momentum and Nesterov's accelerated gradient are accelerated versions of gradient descent. In addition, our algorithms are much better than   subMPC with 1 iteration, implying that our algorithms utilize the lookahead information more efficiently given  similar computational time. Finally,
 subMPC achieves better performance by increasing the iteration number but the improvement saturates as $W$ increases, in contrast to the steady improvement of RHGC.  

\textbf{Path tracking for a two-wheel mobile robot.} Though we presented our online algorithms on an LTI system, our RHGC methods are applicable to some nonlinear systems as well. Here we consider a two-wheel mobile robot with nonlinear kinematic dynamics  $\dot x = v\cos \delta, \dot y = v \sin \delta, \dot \delta = w$
where $(x,y)$ is the robot location, $v$ and  $w$ are the tangential and angular velocities respectively, $\delta$ denotes the tangent angle between $v$ and the $x$ axis \cite{klancar2005mobile}. The control is directly on the $v$ and $w$, e.g., via the pulse-width modulation (PWM) of the motor \cite{mpi}.
Given a reference path $(x^r_t, y^r_t)$, the objective is to balance the tracking performance and the control cost, i.e., $\min \ \  \sum_{t = 0}^N \left[ c_t\cdot\left((x_t - x^r_t)^2 + (y_t - y^r_t)^2 \right)+c^v_t\cdot (v_t)^2 + c^w_t\cdot (w_t)^2\right]$.
We discretize the dynamics with time interval $\Delta t=0.025\text{s}$; then  follow similar ideas in this paper to reformulate the optimal path tracking problem  to an unconstrained optimization with respect to $(x_t,y_t)$ and apply RHGC. See Appendix H for details. Figure \ref{fig: robot} plots the tracking results with window $W=40$ and $W=80$ corresponding to lookahead time $1$s and $2$s. A video showing the dynamic processes with different $W$ is provided at \url{https://youtu.be/fal56LTBD1s}. It is observed that the robot follows the reference trajectory well especially when the path is smooth but deviates a little more when the path has sharp turns, and a longer lookahead window leads to better tracking performance.  These results confirm that our RHGC works effectively on nonlinear systems.







\section{Conclusion}
This paper studies the role of predictions on dynamic regrets of online control problems with linear dynamics. We design RHGC algorithms and provide  regret upper bounds of two specific algorithms: RHGD and RHTM. We also provide a fundamental limit and show the fundamental limit almost matches RHTM's upper bound. \black{This paper leads to many interesting future directions, some of which are briefly discussed below. The first direction is to study more realistic prediction models which considers random prediction noises, e.g. \cite{chen2015online, chen2016using,bang2005doubly}.  The second direction is to consider unknown systems with process noises, possibly by applying learning-based control tools \cite{dean2017sample,tu2017least,ouyang2017learning}. Further, more studies could be conducted on general control problems including nonlinear control and control with input and state constraints. Besides, it is interesting to consider other performance metrics, such as competitive ratio, since the dynamic regret is non-vanishing. Finally, other future directions include closing the gap of the regret bounds and more discussion on the effect of the canonical-form transformation on the condition number.}




\section*{Acknowledgement}
This work was supported by NSF Career 1553407, ARPA-E NODES, AFOSR YIP and ONR YIP programs.

\bibliographystyle{unsrt}

\bibliography{citation4OC}
\newpage
\appendix
\section*{Appendices}

In Appendix \ref{append: canonical form}, we will discuss the canonical-form transformation. In Appendix \ref{append: triple momentum}, we will briefly introduce the  triple momentum algorithm proposed in \cite{van2017fastest} and provide the proof of Theorem \ref{thm: RHGM regret bdd}. In Appendix \ref{append: C(z) property}, we will provide the proof of Lemma \ref{lem: C(z) properties}. In Appendix \ref{append: initial}, we will present the proof of Theorem \ref{thm: initialization}.  Appendix \ref{append: LQT} provides a proof of Corollary \ref{cor: LQT Q, R not change} and regret analysis for more general linear quadratic tracking problems. Appendix \ref{append: lower bound} provides a proof of Theorem \ref{thm: lower bdd}. In Appendix \ref{append: LQT proof}, we will provide the proofs of the technical lemmas used in Appendix \ref{append: LQT}. In Appendix \ref{append: simulation}, we will provide  a more detailed description of our numerical experiments.

\section{Canonical form}\label{append: canonical form}
In this section, we introduce the linear transformation from a general LTI system to a canonical-form LTI system, and then discuss how to convert a general online optimal control problem to an online optimal control problem with a canonical-form system.

Firstly, consider a general LTI system: $x_{t+1}=Ax_t +Bu_t$ and two invertible matrices $S_x\in \R^n, S_u\in R^m$. Under the linear transformation on state and control: $\hat x_t= S_x x_t, \ \hat u_t = S_u u_t$, the equivalent LTI system with respect to the new state $\hat x_t$ and the new control $\hat u_t$ is
\begin{align*}
    \hat x_{t+1}= S_x A S_x^{-1}\hat x_t + S_x B S_u^{-1}\hat u_t
\end{align*}

By   \cite[Theorem 1]{luenberger1967canonical}, for any controllable $(A, B)$, there exist $S_x, S_u$ such that $\hat A = S_x A S_x^{-1}$ and $\hat B= S_x B S_u^{-1}$ are in the canonical form defined in Definition \ref{def: canonical form}. The computation method of $S_x, S_u$ is also provided in \cite{luenberger1967canonical}.

In an online optimal control problem, since $A,B$ are known as priors, $S_x, \ S_u$ can be computed offline. When stage cost functions $f_t(x_t), g_t(u_t)$ are received online, the new cost functions $\hat f_t(\hat x_t), \hat g_t(\hat u_t)$ for the canonical-form system can be computed online by applying $S_x,S_u$:
\begin{align*}
\hat f_t(\hat x_t)= f_t(x_t)=f_t(S_x^{-1}\hat x_t),\quad \hat g_t(\hat u_t)= g_t(u_t)=g_t(S_u^{-1}\hat u_t)
\end{align*}
{Moreover, it is straightforward to verify that $\hat f_t(\hat x_t)$ and $\hat g_t(\hat u_t)$ still satisfy Assumption \ref{ass: general Qt, Rt factors} and \ref{ass: theta xi bdd}, just with perhaps different parameters. For example, $\hat f_t(\hat x_t)$ is $\mu_f/\|S_x\|^2$ strongly convex and $l_f \|S_x^{-1}\|^2$ smooth and $\hat g_t(\hat u_t)$ is  convex  and $l_g \|S_u^{-1}\|^2$ smooth.  } Therefore, it is without loss of generality to only consider online optimal control with canonical-form systems.

\section{Triple momentum and a proof of Theorem \ref{thm: RHGM regret bdd}}\label{append: triple momentum}

Triple Momentum (TM) is an accelerated version of gradient descent proposed in \cite{van2017fastest}. When optimizing an unconstrained optimization $\min_{\mbf z}C(\mbf z)$, at each iteration $j\geq 0$, TM conducts
\begin{align*}
& \pmb\omega(j+1)= (1+\delta_{\omega}) \pmb\omega(j)-\delta_{\omega}\pmb\omega_(j-1) -\delta_c \nabla C(\mbf y(j))\\
& \mbf y(j+1) = (1+\delta_y) \pmb\omega(j+1) -\delta_y\pmb\omega(j)\\
& \mbf z(j+1)=(1+\delta_z)\pmb\omega(j+1)-\delta_z \pmb\omega(j)
\end{align*}
where $\pmb\omega(j), \mbf y(j)$ are auxiliary variables to accelerate the convergence, $\mbf z(j)$ is the decision variable, $\pmb \omega(0)=\pmb \omega(-1)=\mbf z(0)=\mbf y(0)$ are given initial values.

Suppose $\mbf z=(z_1^\top , \dots, z_N^\top )^\top $. Zooming in to each coordinate $z_t$, the update of $z_t(j)$ by TM is provided below
\begin{align*}
&\omega_{t}(j+1)= (1+\delta_{\omega}) \omega_t(j)-\delta_{\omega}\omega_{t}(j-1) -\delta_c  \frac{\partial C}{\partial y_t}(\mbf y(j))\\
& y_t(j+1) = (1+\delta_y) \mbf \omega_t(j+1) -\delta_y \omega_t(j)\\
& z_t(j+1)=(1+\delta_z)\omega_t(j+1)-\delta_z \omega_{t}(j)
\end{align*}
By Section \ref{sec: RHTM},  $\frac{\partial C}{\partial y_t}(\mbf y(j))$ only depends on stage cost functions and  stage variables  across a finite neighboring stages, 
allowing the online implementation in Algorithm \ref{alg:RHTM-LQT} based on the finite-lookahead window.

TM enjoys a faster convergence rate than the gradient descent for $\mu_c$ strongly convex and $l_c$ smooth functions under proper step sizes. In particular, when $\gamma_c = \frac{1+\phi}{l_c}, \ \gamma_w = \frac{\phi^2}{2-\phi}, \ \gamma_y = \frac{\phi^2}{(1+\phi)(2-\phi)}, \ \gamma_z = \frac{\phi^2}{1-\phi^2}$, and $\phi= 1-1/\sqrt{\zeta}$, $\zeta= l_c/\mu_c$, by \cite[Theorem 1]{van2017fastest}, the convergence of TM satisfies:
\begin{align}\label{equ: TM convergence}
    C(\mbf z(j))-C(\mbf z^*) \leq (\frac{\sqrt{\zeta}-1}{\sqrt \zeta})^{2j} \frac{l_c \zeta}{2} \| \mbf z(0)-\mbf z^* \|^2 \leq \zeta^2 (\frac{\sqrt{\zeta}-1}{\sqrt \zeta})^{2j}(C(\mathbf z(0))-C(\mathbf z^*))
\end{align}
In the following, we will apply this result to prove Theorem \ref{thm: RHGM regret bdd}.

\subsection{Proof of Theorem \ref{thm: RHGM regret bdd}}

By comparing TM with RHTM, it can be verified that $z_{t+1}(K)$ computed by RHTM is the same as  $z_{t+1}(K)$ computed by the triple momentum after  $K$ iterations. Moreover,
	according to  the equivalence between  the optimization  $\min_{\mbf z} C(\mbf z)$ and the optimal control $J(\mbf x, \mbf u)$ in Lemma \ref{lem: C(z) properties},  $$J(RHTM)=C(\mbf z(K)), \ J(\varphi)=C(\mbf z(0)), \ J^*=C(\mbf z^*)$$
	Finally, by utiltizing \eqref{equ: TM convergence}, the bound on $\text{Regret}(RHTM)$ is straightforward.

{
The regret of RHGD can be proved in the same way.
}

\section{Proof of Lemma \ref{lem: C(z) properties}}\label{append: C(z) property}
Property ii) and iii) can be directly verified by definition. Thus, it suffices to prove i): the  strong convexity and smoothness of $C(\mbf z)$.

Notice that $x_t, \ u_t$ are linear with respect to $\mbf z$ by \eqref{equ: xt's representation by z} \eqref{equ: ut= zt+1- mr Axt}. For ease of reference, we define matrix $M^{x_t}, M^{u_t}$ to represent the relation between $x_t,u_t$ and $\mbf z$, i.e,  $x_t = M^{x_t} \mbf z$ and $u_t = M^{u_t} \mbf z$. Similarly, we write $\tilde f_t(z_{t-p+1}, \dots, z_t)$ and $\tilde g_t(z_{t-p+1}, \dots, z_{t+1})$ in terms of $\mbf z$ for simplicity of notation:
\begin{align*}
    \tilde f_t(z_{t-p+1}, \dots, z_t)&= \tilde f_t(\mbf z) = f_t(M^{x_t}\mbf z)\\
    \tilde g_t(z_{t-p+1}, \dots, z_{t+1})&= \tilde g_t(\mbf z) = g_t(M^{u_t}\mbf z)
\end{align*}

A direct consequence of the linear relations is that $\tilde f_t(\mbf z)$ and $\tilde g_t(\mbf z)$ are convex with respect to $\mbf z$ because $f_t(x_t), g_t(u_t)$ are convex and the linear transformation preserves convexity. 

In the following, we will focus on the proof of strong convexity and smoothness. For simplicity, in the following, 	 we  only consider cost function $f_t, g_t$ with minimum values  zero: $f_t(\theta_t)=0$, and $g_t(\xi_t)=0$ for all $t$. This is  without loss of generality because by strong convexity and smoothness, $f_t$, $g_t$ have minimum values, and by subtracting the minimum value, we can let $f_t, g_t$ have minimum value 0.

\paragraph{Strong convexity.} Since $\tilde g_t$ is convex, we only need to prove that $\sum_t \tilde f_t(\mbf z)$ is strongly convex then the sum $C(\mbf z)$ is strongly convex because the sum of convex functions and a strongly convex function is strongly convex.

In particular, by the strong convexity of $f_t(x_t)$, we have the following result: for any $\mbf z, \mbf z'\in \R^{Nm}$ and $x_t=M^{x_t}\mbf z$, $x_t'=M^{x_t}\mbf z'$:
	\begin{align*}
	\tilde f_t(\mbf z')&- \tilde f_t(\mbf z) - \langle \nabla \tilde f_t(\mbf z), \mbf z' -\mbf z \rangle - \frac{\mu_f}{2}\| z'_t - z_t \|^2\\
	& = \tilde f_t(\mbf z')- \tilde f_t(\mbf z) - \langle (M^{x_t})^\top  \nabla f_t(x_t), \mbf z' -\mbf z \rangle - \frac{\mu_f}{2}\| z'_t - z_t \|^2\\
	& =  \tilde f_t(\mbf z')- \tilde f_t(\mbf z) - \langle  \nabla f_t(x_t), M^{x_t}(\mbf z' -\mbf z) \rangle - \frac{\mu_f}{2}\|z'_t - z_t \|^2\\
	&= \tilde f_t(\mbf z')- \tilde f_t(\mbf z) - \langle  \nabla f_t(x_t),  x'_t -x_t \rangle - \frac{\mu_f}{2}\| z'_t - z_t \|^2   \\
	& \geq f_t(x'_t)-f_t(x_t )- \langle  \nabla f_t(x_t),  x'_t -x_t\rangle - \frac{\mu_f}{2}\| x'_t - x_t \|^2  \geq 0 
	\end{align*}
where the first equality is by the chain rule, the second equality is by the definition of inner product, the third equality is by the definition of $x_t, x_t'$, the first inequality is by $\tilde f_t(z)=f_t(x)$ and $z_t=(x_t^{k_1},\dots, x_t^{k_m})^\top $, and the last inequality is because
$f_t(x_t) $ is $\mu_f$ strongly convex.

Summing over $t$ on both sides of the inequality results in the strong convexity of $\sum_t \tilde f_t(\mbf z)$:
\begin{align*}
	&\sum_{t=1}^{N}\left[\tilde f_t(\mbf z')- \tilde f_t(\mbf z) - \langle \nabla \tilde f_t(\mbf z), \mbf z' -\mbf z \rangle - \frac{\mu_f}{2}\| z'_t - z_t \|^2 \right]  \\
	=& \sum_{t=1}^{N}\tilde f_t(\mbf z')- \sum_{t=1}^{N}\tilde f_t(\mbf z) - \langle \nabla \sum_{t=1}^{N}\tilde f_t(\mbf z), \mbf z' -\mbf z \rangle - \frac{\mu_f}{2}\| \mbf z' - \mbf z \|^2 \geq 0	 
	\end{align*}
	Consequently, $C(\mbf z)$ is  strongly convex with parameter at least $\mu_f$  by the convexity of $\tilde g_t$.

\paragraph{Smoothness.} 
We will prove the smoothness by considering  $ \tilde f_t(\mbf z)$ and $ \tilde g_t(\mbf z)$ respectively. 
	
	Firstly, let's consider $\tilde f_t(\mbf z)$. Similar to the proof for strong convexity, we use the smoothness of $f_t(x_t)$. For any $\mbf z, \mbf z'$, and $x_t=M^{x_t}\mbf z$, $x_t'=M^{x_t}\mbf z'$,  we can show that
	\begin{align*}
	\tilde f_t(\mbf z')& = f_t( x'_t) \leq f_t(x_t)+ \langle \nabla f_t(x_t),  x'_t-x_t\rangle + \frac{l_f}{2}\| x'_t-x_t\|^2\\
& \leq \tilde f_t(\mbf z) + \langle \nabla \tilde f_t(\mbf z), \mbf z'-\mbf z \rangle + \frac{l_f}{2}(\| z'_{t-p+1}-z_{t-p+1}\|^2 + \dots + \| z'_t-z_t\|^2 )
	\end{align*}
where the second inequality is by $x_t=M^{x_t}\mbf z$ and the chain rule and \eqref{equ: xt's representation by z}. 
	
	Secondly, we consider $\tilde g_t(z)$ in a similar way. For any $\mbf z, \mbf z'$, and $u_t=M^{u_t}\mbf z$, $u_t'=M^{u_t}\mbf z'$, we have
	\begin{align*}
 \tilde g_t(\mbf z')	&= g_t( u'_t)\leq g_t(u_t)+  \langle \nabla g_t(u_t),  u'_t-u_t\rangle + \frac{l_g}{2}\|u'_t- u_t\|^2 \\
	& =\tilde g_t(\mbf z)+  \langle  (M^{u_t})^\top  \nabla g_t(u_t), \mbf z'-\mbf z\rangle + \frac{l_g}{2}\|u'_t- u_t\|^2  \\
	& = \tilde g_t(\mbf z)+  \langle \nabla \tilde g_t(\mbf z), \mbf z-\mbf z\rangle + \frac{l_g}{2}\|u'_t- u_t\|^2  
		\end{align*}
		Since $u_t= z_{t+1}-A(\I, :) x_t=[I_m,-A(\I, :)](z_{t+1}^\top ,x_t^\top )^\top $, we have that 
		\begin{align*}
 \frac{l_g}{2}\|u'_t- u_t\|^2  	& \leq  \frac{l_g}{2}\|[I_m,-A(\I, :)]\left[ ( (z_{t+1}')^\top , (x'_t)^\top )^\top  - ( z_{t+1}^\top ,  x^\top _t)^\top  \right] \|^2 \\
	& \leq  \frac{l_g}{2}\|[I_m,-A(\I, :)]\|^2 (\|z_{t+1}- z'_{t+1}\|^2+ \|x_t -  x'_t\|^2)\\
	& \leq  \frac{l_g}{2}\|[I_m,-A(\I, :)]\|^2 (\|z_{t+1}-z'_{t+1}\|^2+ \dots + \|z_{t-p+1}-z'_{t-p+1}\|^2)
	\end{align*}
	Finally, by summing $\tilde f_t(\mbf z'), \tilde g_t(\mbf z')$'s inequalities above over all $t$, we have
		\begin{align*}
	&C(\mbf z') \leq C(\mbf z) + \langle \nabla C(\mbf z), \mbf z'-\mbf z\rangle + (p l_f + (p+1)l_g \|[I_m,-A(\I, :)]\|^2)/2 \|\mbf z'-\mbf z\|^2
	\end{align*}
 Thus, we have proved the smoothness of $C(\mbf z)$.

\section{Proof of Theorem \ref{thm: initialization}}\label{append: initial}
Remember that $\text{Regret}(FOSS) =  J(FOSS)-J^*$. To  bound the regret, we let the sum of the optimal steady state costs, $\sum_{t=0}^{N-1} \lambda^e_t$, be a middle ground and bound $J(FOSS)-\sum_{t=0}^{N-1}\lambda^e_t$ and 
$\sum_{t=0}^{N-1}\lambda^e_t-J^*$ in Lemma \ref{lem: J(varphi)-sum lambda e} and Lemma \ref{lem: sum lambda e -J*} respectively. Then, the regret bound can be obtained by combining the two bounds.
\begin{lemma}[Bound on $J(FOSS)-\sum_{t=0}^{N-1}\lambda_t^e$]\label{lem: J(varphi)-sum lambda e}
 Let  $x_t(0)$ denote the state determined by FOSS. 
\begin{align*}
     J(FOSS)-\sum_{t=0}^{N-1}\lambda^e_t \leq c_1 \sum_{t=0}^{N-1}\|x_t^e -x_{t-1}^e\| + f_N(x_N(0))=O\left(  \sum_{t=0}^{N}\|x_t^e -x_{t-1}^e\|\right)
\end{align*}
   where we define $x_N^e\coloneqq \delta_N$, $x_{-1}^e\coloneqq x_0=0$ for simplicity of  notation, $c_1$ is a constant that does not depend on $N, W$ and big $O$ hides a constant that does not depend on $N, W$.
\end{lemma}

\begin{lemma}[Bound on $\sum_{t=0}^{N-1}\lambda^e_t-J^*$]\label{lem: sum lambda e -J*}
Let $h_t^e(x)$ denote a solution to the  Bellman equations under cost $f_t(x)+g_t(u)$. Let $\{ x_t^*\}$ denote the optimal state trajectory to the offline optimal control \eqref{equ: control problem}. 
\begin{align*}
     \sum_{t=0}^{N-1}\lambda^e_t-J^*\leq \sum_{t=1}^{N}(h^e_{t-1}(x_{t}^*)-h^e_{t}(x_{t}^*))-h_0^e(x_0) = \sum_{t=0}^{N}(h^e_{t-1}(x_{t}^*)-h^e_{t}(x_{t}^*))
\end{align*}
   where we define $h_N^e(x)\coloneqq f_N(x)$, $h_{-1}^e(x)\coloneqq 0$ and $x_0^*\coloneqq x_0$ for simplicity of  notation. 
\end{lemma}

Then, we can complete the proof by applying 
 Lemma \ref{lem: J(varphi)-sum lambda e} and \ref{lem: sum lambda e -J*}:
\begin{align*}
J(FOSS)-J^*& =J(FOSS)-\sum_{t=0}^{N-1}\lambda^e_t + \sum_{t=0}^{N-1}\lambda^e_t-J^*\\
&= O\left(\sum_{t=0}^{N}(\| x_{t-1}^e- x_t^e \|  + h^e_{t-1}(x_{t}^*)-h^e_{t}(x_{t}^*))\right)
\end{align*}

In the following, we will prove Lemma \ref{lem: J(varphi)-sum lambda e} and \ref{lem: sum lambda e -J*} respectively. For simplicity, 	 we  only consider cost function $f_t, g_t$ with minimum values zero: $f_t(\theta_t)=0$, and $g_t(\xi_t)=0$ for all $t$. There is no loss of generality because by strong convexity and smoothness, $f_t$, $g_t$ have minimum values, and by subtracting the minimum value, we can let $f_t, g_t$ have minimum value 0.

\subsection{Proof of Lemma \ref{lem: J(varphi)-sum lambda e}.}

Notice that $J(FOSS)= \sum_{t=0}^{N-1}(f_t(x_t(0))+g_t(u_t(0)) ) + f_N(x_N(0))$ and $\sum_{t=0}^{N-1}\lambda_t^e =\sum_{t=0}^{N-1}(f_t(x_t^e)+g_t(u_t^e) ) $. Thus, it suffices to bound $f_t(x_t(0))-f_t(x_t^e)$ and $g_t(u_t(0))-g_t(u_t^e)$ for $0\leq t \leq N-1$. We will first focus on $f_t(x_t(0))-f_t(x_t^e)$, then bound $g_t(u_t(0))-g_t(u_t^e)$ in the same way.

For $0\leq t\leq N-1$, by the convexity of $f_t$, and the property of $L_2$ norm,
\begin{align}\label{equ: bdd ft - ft by convexty}
    f_t(x_t(0))-f_t(x_t^e)& \leq \langle \nabla f_t(x_t(0)), x_t(0)-x_t^e\rangle \leq \| \nabla f_t(x_t(0)) \| \| x_t(0)-x_t^e \|
\end{align}
In the following, we will bound $\| \nabla f_t(x_t(0)) \|$ and $ \| x_t(0)-x_t^e \|$ respectively. 

Firstly, we provide a bound on $\| \nabla f_t(x_t(0)) \|$:
\begin{align}\label{equ: bdd grad f}
    \| \nabla f_t(x_t(0)) \|= \| \nabla f_t(x_t(0)) -\nabla f_t(\theta_t)\| \leq l_f \| x_t(0)-\theta_t\| \leq l_f(\sqrt n \bar x^e + \bar \theta) 
\end{align}
where the first equality is because $\theta_t$ is the global minimizer of $f_t$, and first inequality is by Lipschitz smoothness, the second inequality is by $\|\theta_t \|\leq \bar \theta$ according to Assumption \ref{ass: theta xi bdd} and by $\| x_t(0)\|\leq \sqrt n \bar x^e\|$ proved in the following lemma. 

\begin{lemma}[Uniform upper bounds on $x_t^e, u_t^e, x_t(0), u_t(0)$]\label{lem: bdd xte, ute, xt(0), ut(0)}
	There exist $\bar x^e$ and $\bar u^e$ that are independent of $N, W$, such that $\|x_t^e\| \leq \bar x^e$ and $\|u_t^e \|\leq \bar u^e$ for all $0\leq t\leq N-1$. Moreover, $\|x_t(0)\| \leq \sqrt{n}\bar x^e$ for $0\leq t \leq N$ and  $\|u_t(0)\| \leq \sqrt n \bar u^e$ for $0\leq t \leq N-1$, where  $x_t(0), u_t(0)$ denote the state and control at $t$ determined by FOSS.
\end{lemma}
The proof is technical and is deferred to Appendix \ref{append: proof of Lem xte bdd}.

Secondly, we provide a bound on $ \| x_t(0)-x_t^e \|$. The proof relies on the expressions of the steady state $x_t^e$ and the initialized state $x_t(0)$ of a canonical-form system.

\begin{lemma}[The steady state and the initialized state of canonical-form systems]\label{lem: x e x(0) form}
    Consider a canonical-form system: $x_{t+1}=Ax_t+B u_t$. 
    \begin{enumerate}
        \item[(a)] Any steady state $(x,u)$ is in the form of 
        \begin{align*}
x&=(\underbrace{z^1, \dots, z^1}_{p_1}, \underbrace{z^2, \dots, z^2}_{p_2}, \dots, \underbrace{z^m, \dots, z^m}_{p_m})^\top \\
u& = (z^1, \dots, z^m)^\top -A(\I, :) x  
\end{align*}
for some $z^1, \dots, z^m \in \R$.
Let $z=(z^1, \dots, z^m)^\top $. 
For the optimal steady state with respect to cost $f_t+g_t$, we denote the corresponding $z$ as $z_t^e$, and the optimal steady state can be represented as
 $x_t^e=(z_t^{e,1}, \dots, z_t^{e,1}, z_t^{e,2}, \dots, z_t^{e,2}, \dots, z_t^{e,m}, \dots, z_t^{e,m})^\top $ and $u_t^e= z_t^e - A(\I, :) x_t^e$ for $0\leq t \leq N-1$. 
\item[(b)] By FOSS initialization, $z_{t+1}(0)=z_t^e$, and $x_t(0)$, $u_t(0)$ satisfy
\begin{align*}
    x_t(0)&=(\underbrace{z_{t-p_1}^{e,1}, \dots, z_{t-1}^{e,1}}_{p_1}, \underbrace{z_{t-p_2}^{e,2},\dots, z_{t-1}^{e,2}}_{p_2}, \dots, \underbrace{z_{t-p_m}^{e,m},\dots, z_{t-1}^{e,m}}_{p_m}), & 0 \leq t \leq N\\
    u_t(0) & = z_t^e -A(\I, :) x_t(0)& 0 \leq t \leq N-1
\end{align*}
where $z_t^e=0$ for $t\leq -1$.
\end{enumerate}
\end{lemma}
\begin{proof}
    \begin{enumerate}
        \item[(a)] This is by the definition of the canonical form and the definition of the steady state.
        \item[(b)] By the initialization, $z_t(0)=x^{e,\I}_{t-1}= z_{t-1}^e$. By the relation between $z_t(0)$ and $x_t(0)$, $u_t(0)$, we have $x_t^{\I}(0)= z_t(0)=z_{t-1}^e$, and $x_t^{\I-1}(0)=z_{t-1}(0)=z_{t-2}^e$, so on and so forth. This proves the structure of $x_t(0)$. The structure of $u_t(0)$ is  because $u_t(0)=z_{t+1}(0)-A(\I, :) x_t(0)= z_t^e -A(\I, :) x_t(0)$
    \end{enumerate}
\end{proof}


By Lemma \ref{lem: x e x(0) form}, we can bound $ \| x_t(0)-x_t^e \|$ for $0\leq t \leq N-1$ by 
\begin{align}
     \| x_t(0)-x_t^e \|& \leq  \sqrt{ \|z_{t-1}^e-z_t^e\|^2 + \dots + \| z_{t-p}^e - z_t^e\|^2} \notag\\
     &\leq \sqrt{\|x_{t-1}^e-x_t^e\|^2 + \dots + \| x_{t-p}^e - x_t^e\|^2}\notag\\
     & \leq \|x_{t-1}^e-x_t^e\| + \dots + \| x_{t-p}^e - x_t^e\| \notag\\
     &\leq p(\|x_{t-1}^e-x_t^e\| + \dots + \| x_{t-p}^e - x_{t-p+1}^e\|) \label{equ: bdd xt(0)-xte}
\end{align}

Combining \eqref{equ: bdd ft - ft by convexty} \eqref{equ: bdd grad f} and \eqref{equ: bdd xt(0)-xte} yields
\begin{align}
   \sum_{t=0}^{N-1} f_t(x_t(0))-f_t(x_t^e)& \leq \sum_{t=0}^{N-1} \| \nabla f_t(x_t(0)) \| \| x_t(0)-x_t^e \|\notag \\
   &\leq \sum_{t=0}^{N-1}l_f(\sqrt n \bar x^e+\bar \theta ) p(\|x_{t-1}^e-x_t^e\| + \dots + \| x_{t-p}^e - x_{t-p+1}^e\|) \notag\\
   & \leq p^2 l_f(\sqrt n \bar x^e+\bar \theta ) \sum_{t=0}^{N-1}\|x_{t-1}^e-x_t^e\| \label{equ: sum of ft -ft's bdd}
\end{align}
Notice that the constant term $p^2 l_f(\sqrt n \bar x^e+\bar \theta )$ does not depend on $N, W$.

Similarly, we can provide a bound on $g_t(u_t(0))-g_t(u_t^e)$.
\begin{align}
    \sum_{t=0}^{N-1} g_t(u_t(0))-g_t(u_t^e)& \leq \sum_{t=0}^{N-1} \| \nabla g_t(u_t(0))\| \| u_t(0)-u_t^e\|  \notag\\
    & \leq \sum_{t=0}^{N-1} l_g \| u_t(0)-\xi_t\| \| u_t(0)-u_t^e\| \notag \\
    & \leq \sum_{t=0}^{N-1} l_g(\sqrt n \bar u^e+\bar \xi) \| A(\I, :) x_t(0)-A(\I, :) x_t^e\|\notag\\
    & \leq \sum_{t=0}^{N-1} l_g(\sqrt n \bar u^e+\bar \xi)\|A(\I, :)\| \|x_t(0)-x_t^e \| \notag\\
    & \leq p^2l_g(\sqrt n \bar u^e+\bar \xi)\|A(\I, :)\|\sum_{t=0}^{N-1}\|x_{t-1}^e-x_t^e\| \label{equ: sum of gt -gt's bdd}
\end{align}
where the first inequality is by the convexity, the second inequality is because $\xi_t$ is the global minimizer of $g_t$ and $g_t$ is $l_g$-smooth, the third inequality is by Assumption \ref{ass: theta xi bdd}, Lemma \ref{lem: bdd xte, ute, xt(0), ut(0)} and  Lemma \ref{lem: x e x(0) form}, the fifth inequality is by  \eqref{equ: bdd xt(0)-xte}. Notice that the constant term $p^2l_g(\sqrt n \bar u^e+\bar \xi)\|A(\I, :)\|$ does not depend on $N, W$.

By \eqref{equ: sum of ft -ft's bdd} and \eqref{equ: sum of gt -gt's bdd}, we complete the proof of the first inequality in the statement of Lemma \ref{lem: J(varphi)-sum lambda e}:
\begin{align*}
    J(FOSS)-\sum_{t=0}^{N-1}\lambda_t^e & \leq c_1 \sum_{t=0}^{N-1}\|x_{t-1}^e-x_t^e\| + f_N(x_N(0))
\end{align*}
where $c_1$ does not depend on $N$, $W$.

By defining $x_N^e = \theta_N$, we can bound $f_N(x_N(0))$ by $\| x_N(0)-x_N^e\|$ up to some constants because $f_N(x_N(0))=f_N(x_N(0))-f_N(\theta_N)  \leq \frac{l_f}{2}(\sqrt n \bar x^e+ \bar \theta) \| x_N(0)-x_N^e\|$.  By the same argument as in \eqref{equ: bdd xt(0)-xte}, we have $\| x_N(0)-x_N^e\| = O(\sum_{t=0}^{N}\| x_{t-1}^e- x_t^e \|)$, where the big $O$ hides some constant that does not depend on $N$, $W$. Consequently, 
 $$J(FOSS)-\sum_{t=0}^{N-1}\lambda_t^e =O\left(  \sum_{t=0}^{N}\|x_{t-1}^e-x_t^e\| \right)$$
 
 \qed

 \subsection{Proof of Lemma \ref{lem: sum lambda e -J*}.}
 The proof heavily relies on dynamic programming and the Bellman equations. For simplicity, we introduce a Bellman  operator  $\B(f+g, h)$ defined by $\B(f+g, h)(x)= \min_u (f(x)+g(u)+h(Ax+Bu))$. Now the Bellman equations can be written as $\B(f+g, h^e)(x)=h^e(x)+\lambda^e$ for any $x$.

 We define a sequence of auxiliary functions $S_k$:  $S_k(x)= h_k^e(x)+\sum_{t=k}^{N-1} \lambda_t^e$ for $k=0, \dots, N$, where  $h_N^e(x)=f_N(x)$. 
 
 We first provide a recursive equation for $S_k$. By Bellman equations, we have $h_k^e(x)+\lambda_k^e=\B(f_k+g_k, h_k^e)(x)$ for $0\leq k \leq N-1$. Let $\pi^e_k$ be the corresponding optimal control policy that solves the Bellman equations.  We have the following recursive relation for $S_k$  when $0\leq k \leq N-1$:
 \begin{align*}
     S_k&(x) = \B(f_k+g_k, S_{k+1}-h^e_{k+1}+h_k^e)(x) 
 \end{align*}
 where $S_N(x)=f_N(x)$.
 
 Further, 	let $V_k(x)$ denote the optimal cost-to-go function from $k$ to $N$, then we  obtain a recursive equation for $V_k$ by dynamic programming:
 \begin{align*}
     V_k&(x) = \B(f_k+g_k, V_{k+1})(x)= f_k(x)+g_k(\pi^*_k(x)) + V_{k+1}(Ax + B \pi^*_k(x))
 \end{align*}
	where $0\leq k \leq N-1$, and $\pi^*_k$ denotes the optimal control policy and $V_N(x)=f_N(x)$.
	
	Now, we are ready for a recursive inequality for $S_k(x_k^*)-V_k(x_k^*)$. Let $\{x_k^*\}$ denote the optimal trajectory, then $x_{k+1}^* = Ax_k^*+B\pi^*_k(x_k^*)$. For any $k=0,\dots, N-1$,  
	\begin{align*}
	S_k(x_k^*)-V_k(x_k^*)&= \B(f_k+g_k, S_{k+1}-h^e_{k+1}+h^e_k)(x_k^*)-\B(f_k+g_k, V_{k+1})(x_k^*)\\
	& \leq  f_k(x_k^*)+g_k(\pi^*_k(x_k^*)) + S_{k+1}(x_{k+1}^*)-h^e_{k+1}(x_{k+1}^*)+h_k^e(x_{k+1}^*)\\
	&  - (f_k(x_k^*)+g_k(\pi^*_k(x_k^*)) + V_{k+1}(x_{k+1}^*))\\
	 &=  S_{k+1}(x_{k+1}^*)-h^e_{k+1}(x_{k+1}^*)+h_k^e(x_{k+1}^*)-V_{k+1}(x_{k+1}^*)
	\end{align*}
	where the first inequality is because $\pi_k^*$ is not optimal for the Bellman operator $\B(f_k+g_k, S_{k+1}-h^e_{k+1}+h^e_k)(x_k^*)$.
	
	Summing over $k=0, \dots, N-1$ the recursive inequality for $S_k(x_k^*)-V_k(x_k^*)$ yields
	\begin{align*}
	&S_0(x_0)-V_0(x_0)\leq \sum_{k=0}^{N-1}(h^e_k(x_{k+1}^*)-h^e_{k+1}(x_{k+1}^*))
	\end{align*}
	By subtracting $h_0^e(x_0)$ on both sides,
	\begin{align*}
	&\sum_{t=0}^{N-1}\lambda^e_t-J^*\leq \sum_{k=0}^{N-1}(h^e_k(x_{k+1}^*)-h^e_{k+1}(x_{k+1}^*))-h^e_0(x_0)
	\end{align*}
	For the simplicity of notation, we define $h^e_{-1}(x_0)=0$ and $x_0^*=x_0$, then the bound can be written as
	\begin{align*}
	&\sum_{t=0}^{N-1}\lambda^e_t-J^*\leq  \sum_{k=0}^{N}(h^e_{k-1}(x_{k}^*)-h^e_{k}(x_{k}^*))
	\end{align*}
 
\qed

\subsection{Proof of Lemma \ref{lem: bdd xte, ute, xt(0), ut(0)}}\label{append: proof of Lem xte bdd}
The proof relies on the (strong) convexity and smoothness of the cost functions and the uniform upper bounds on $\theta_t, \xi_t$.

	First of all, suppose there exists $\bar x^e$ such that $\|x_t^e\|_2 \leq \bar x^e$ for all $0\leq t \leq N-1$. We will bound $u_t^e, x_t(0),u_t(0)$ by using $\bar x^e$. Notice that the optimal steady state and the corresponding steady control satisfy: $u_t^e= x_t^{e,\I}- A(\I, :)  x_t^e$. If we can bound $x_t^e$ by $\| x_t^e \| \leq \bar x^e$ for all $t$, $u_t^e$ can be bounded accordingly:
	\begin{align*}
	\| u_t^e\|& \leq \| x_t^{e,\I}\| + \| A(\I, :) x_t^e\|\leq \| x_t^e\| + \| A(\I, :)\| \|x_t^e \|\leq (1+ \|A(\I, :)\|)\bar x^e \eqqcolon \bar u^e
	\end{align*} 
	Moreover, $x_t(0)$ can also be bounded by $\bar x^e$ multiplied by some factors, because by Lemma \ref{lem: x e x(0) form}, $x_t(0)$'s each entry is determined by some entry of $x_s^e$ for $s< t$.  As  a result, for $0\leq t \leq N$
	\begin{align*}
	\| x_t(0)\|_2 \leq \sqrt n \| x_t(0)\|_\infty \leq \sqrt n \max_{s< t}\|x_s^e \|_\infty \leq \sqrt n\max_{s< t}\|x_s^e \|_2\leq  \sqrt n \bar x^e
	\end{align*}
	We can bound $u_t(0)$ by  noticing that $u_t(0)= x_{t+1}^{\I}(0)- A(\I, :) x_t(0)$ and 
	\begin{align*}
\| u_t(0)\|& \leq \| x_{t+1}^{\I}(0)\| + \| A(\I, :) x_t(0)\|\leq \| x_{t+1}(0)\| + \| A(\I, :)\| \|x_t(0)\|\\
&\leq (1+ \|A(\I, :)\|)\sqrt n\bar x^e= \sqrt n\bar u^e
\end{align*}

Next, it suffices to  prove $\|x_t^e\| \leq \bar x^e$ for all $t$ for some $\bar x^e$. To prove this bound, we construct another (suboptimal) steady state: $\hat x_t= (\theta^1_t, \dots, \theta^1_t)$. Let $\hat u_t= \hat x_t^{\I}-A(\I, :) \hat x_t$. It can be easily verified that $(\hat x_t, \hat u_t)$ is indeed a steady state of the canonical-form system. Moreover, $\hat x_t$ and $\hat u_t$ can be bounded similarly as follows.
\begin{align*}
\| \hat x_t \| &\leq \sqrt n | \theta_t^1| \leq \sqrt n \| \theta_t \|_\infty \leq \sqrt n \| \theta_t \| \leq \sqrt n \bar \theta \\
\| \hat u_t\|_2& \leq (1+ \|A(\I, :)\|) \| \hat x_t \| \leq (1+ \|A(\I, :)\|)\sqrt n \bar \theta  
\end{align*}

Now, we can bound $\|x_t^e-\theta_t\|$.
\begin{align*}
\frac{\mu}{2}\| x_t^e - \theta_t \|^2 & \leq f_t(x_t^e)- f_t(\theta_t) + g_t(u_t^e)-g_t(\xi_t)  \\
& \leq f_t(\hat x_t)- f_t(\theta_t) + g_t(\hat u_t)-g_t(\xi_t)  \\
& \leq \frac{l_f}{2} \| \hat x_t -\theta_t\|^2 + \frac{l_g}{2} \|\hat u_t- \xi_t\|^2 \\
& \leq l_f (\| \hat x_t \|^2 + \| \theta_t \|^2) + l_g (\| \hat u_t \|^2 + \|\xi_t\|^2) \\
& \leq l_f (n \bar \theta^2 + \bar \theta^2) + l_g(( (1+ \|A(\I, :)\|)\sqrt n \bar \theta  )^2 + \bar \xi)  \eqqcolon c_5
\end{align*}
where the first inequality is by $f_t$'s strong convexity and $g_t$'s convexity, the second inequality is because $(x_t^e, u_t^e)$ is an optimal steady state, the third inequality is by the smoothness and $\nabla f_t(\theta_t)=\nabla g_t(\xi_t)=0$, the last inequality is by the bounds of $\|\hat x_t\|, \|\hat u_t\|$, $\theta_t,$ and $\xi_t$.

As a result, we have $\| x_t^e - \theta_t\|\leq \sqrt{ 2c_5/\mu}$. Then,  we can bound $x_t^e$ by $\| x_t^e\| \leq \| \theta_t \| + \sqrt{ 2c_5/\mu} \leq  \bar \theta + \sqrt{ 2c_5/\mu}\eqqcolon \bar x^e$ for all $t$. It can be verified that $\bar x^e$ does not depend on $N, W$.

\qed

\section{Linear quadratic tracking}\label{append: LQT}
In this section, we will provide a regret bound in Corollary \ref{cor: LQT regret upper bdd} for the general LQT  defined in Example \ref{example: LQT}. Based on this, we prove Corollary \ref{cor: LQT Q, R not change}, which is a special case when $Q_t, R_t$ are not changing.

\subsection{Regret bound on the general online LQT problems}
Before the regret bound, we provide an important lemma to characterize the solution to the Bellman equations of the LQT problem.
\begin{lemma}\label{lem: he formula}
One solution to the Bellman equations with stage cost $\frac{1}{2}(x-\theta)^\top Q(x-\theta)+\frac{1}{2}u^\top Ru$ can be represented by
\begin{equation}\label{equ: he form}
    h^e(x)= \frac{1}{2}(x-\beta^e)^\top P^e(x-\beta^e)
\end{equation}
where $P^e$ denotes the solution to the discrete-time algebraic Riccati equation (DARE) with respect to $Q, R, A, B$
\begin{equation}\label{equ: DARE}
     P^e=Q+A^\top (P^e-P^eB(B^\top  P^eB+R)^{-1}B^\top  P^e)A 
\end{equation}
and $\beta^e= F\theta$ where $F$ is a matrix determined by $A, B, Q, R$. 
\end{lemma}
The proof is in Appendix \ref{append: LQT proof}.

For simplicity of notation,  let $P^e(Q, R)$ denote the solution to the DARE under the parameters $Q, R, A, B$ and $F(Q, R)$ denote the matrix in $\beta^e=F\theta$ given parameters $Q, R, A, B$. Here we omit $A, B$ in the arguments of the functions because they will not change in this paper.

In addition, we introduce the following useful notations: $\underline Q=\mu_f I_n, \bar Q= l_f I_n, \underline R= \mu_g I_m, \bar R=l_g I_m$ for $\mu_f, \mu_g >0$, $0<l_f, l_g<+\infty$; and $\bar P=P^e(\bar Q, \bar R)$ and $\underline P=P^e(\underline Q, \underline R)$. Based on the notations above, we define some sets of matrices to be used later:
	\begin{align*}
	&\Q=\{Q\mid  \underline Q \leq Q \leq \bar Q\},\\
	&\Rc= \{R \mid  \underline R \leq R \leq \bar R \},\\
	&\Pc=\{P \mid  \underline P \leq P \leq \bar P\}.
	\end{align*}

Now, we are ready for the regret bound for the general LQT problem.
\begin{corollary}[Bound on general LQT]\label{cor: LQT regret upper bdd}
	Consider the LQT problem in Example \ref{example: LQT}. Suppose for $t=0, 1, \dots, N-1$, the cost matrices satisfy $Q_t \in \Q$, $R_t \in \Rc$. Suppose the terminal cost function satisfies $Q_N\in \Pc$.\footnote{This additional condition is for technical simplicity and can be removed.}
	Then, the regret of RHTM with initialization FOSS can be bounded by
	\begin{align*}
& \textup{Regret}(RHTM) = O\left( \zeta^2(\frac{\sqrt{\zeta}-1}{\sqrt \zeta})^{2K}\left(\sum_{t=1}^N(\| P^e_t-P^e_{t-1}\| +  \|   \beta_t^e-\beta^e_{t-1}\|) + \sum_{t=0}^{N}\| x_{t-1}^e- x_t^e \|  \right)\right)
\end{align*}
	where  $K= \floor{(W-1)/p}$, $x_{-1}^e=x_0$, $x_N^e=\theta_N$,   $\zeta$ is the condition number of the corresponding $C(\mbf z)$, $(x_t^e, u_t^e)$ is the  optimal steady state under cost $Q_t, R_t, \theta_t$, $P^e_t=P^e(Q_t, R_t)$ and $\beta_t^e= F(Q_t, R_t)\theta_t$ for $t=0, \dots, N-1$ and $\beta_N^e=\theta_N$, $P_N^e = Q_N$.
\end{corollary}

\begin{proof}
Before the proof, we  introduce some supportive lemmas on the uniform bounds of $P^e_t, \beta_t^e, x_t^*$ respectively. The intuition behind these uniform bounds is that the cost function coefficients $Q_t, R_t, \theta_t$ are all uniformly bounded by Assumption \ref{ass: general Qt, Rt factors} and \ref{ass: theta xi bdd}.
The proofs are technical and deferred to Appendix \ref{append: LQT proof}.
	
		\begin{lemma}[Upper bound on $x_t^*$]\label{lem: bound xt*}
	For  any $Q_t \in \Q, R_t \in \Rc, Q_N \in \Pc$, there exists $\bar x$ that does not depend on $t$, $N, W$, such that 
	$$\|x_t^*\|_2 \leq  \bar x, \qquad \forall \ 0 \leq t \leq N. $$
\end{lemma}

\begin{lemma}[Upper bound on $\beta^e$]\label{lem: uniform bdd on beta e}
  For any  $Q\in \Q, R\in \Rc$, any $\|\theta\|\leq \bar \theta$, there exists $\bar \beta\geq 0$ that does not depend on $N$ and only depends on $A, B, l_f, \mu_f, l_g, \mu_g, \bar \theta$, such that $\max(\bar \theta,\|\beta^e \|)\leq \bar \beta$, where $\beta^e$ is defined in Lemma \ref{lem: he formula}.
  \end{lemma}
  \begin{lemma}[Upper bound on  $P^e$]\label{lem: uniform bdd on  Pe}
  For any  $Q\in \Q, R\in \Rc$, we have $P^e=P^e(Q, R) \in \Pc$. Consequently, $\| P^e\| \leq \upsilon_{\max}(\bar P)$, where $\upsilon_{\max}(\bar P)$ denotes the largest eigenvalue of $\bar P$.
\end{lemma}

Now, we are ready for the proof of Corollary \ref{cor: LQT regret upper bdd}. 

By Theorem \ref{thm: initialization}, we only need to bound $\sum_{t=0}^{N}(h^e_{t-1}(x_{t}^*)-h^e_{t}(x_{t}^*))$. 
By definition, $P_N^e=Q_N, \beta_N^e=\theta_N$, $h_N^e(x)=f_N(x)$, so we can write $h_t^e(x)= \frac{1}{2}(x-\beta_t^e)^\top  P_t^e(x-\beta_t^e)$ for $0 \leq t \leq N$.

For $0\leq t \leq N-1$, we split $h_t^e(x_{t+1}^*)-h_{t+1}^e(x_{t+1}^*)$ into two parts.
\begin{align*}
h_t^e(x_{t+1}^*)-h_{t+1}^e(x_{t+1}^*)&=\frac{1}{2}(x_{t+1}^*-\beta_t^e)^\top P_t^e(x_{t+1}^*-\beta_t^e)-\frac{1}{2}(x_{t+1}^*-\beta_{t+1}^e)^\top P_{t+1}^e(x_{t+1}^*-\beta_{t+1}^e)\\
  &= \underbrace{\frac{1}{2}(x_{t+1}^*-\beta_t^e)^\top P_t^e(x_{t+1}^*-\beta_t^e)-\frac{1}{2}(x_{t+1}^*-\beta_{t+1}^e)^\top P_{t}^e(x_{t+1}^*-\beta_{t+1}^e)}_{\text{Part 1}}\\
& + \underbrace{\frac{1}{2}(x_{t+1}^*-\beta_{t+1}^e)^\top P_{t}^e(x_{t+1}^*-\beta_{t+1}^e)-\frac{1}{2}(x_{t+1}^*-\beta_{t+1}^e)^\top P_{t+1}^e(x_{t+1}^*-\beta_{t+1}^e)}_{\text{Part 2}} 
\end{align*}
Part 1 can be bounded by the following
\begin{align*}
\text{Part 1} & = \frac{1}{2}(x_{t+1}^*-\beta_t^e+  x_{t+1}^*-\beta_{t+1}^e )^\top P_t^e(x_{t+1}^*-\beta_t^e- (x_{t+1}^*-\beta_{t+1}^e))\\
& \leq \frac{1}{2}\| x_{t+1}^*-\beta_t^e+  x_{t+1}^*-\beta_{t+1}^e \|_2\|P_t^e\|_2 \| \beta_{t+1}^e-\beta_t^e\|_2 \\
& \leq (\bar x + \bar \beta)\upsilon_{max}(\bar P) \| \beta_{t+1}^e-\beta_t^e\|_2 
\end{align*}
where the last inequality is by Lemma \ref{lem: bound xt*}, \ref{lem: uniform bdd on beta e} \ref{lem: uniform bdd on  Pe}.

Part 2 can be bounded by the following when $0\leq t \leq N-1$,
\begin{align*}
\text{Part 2} & = \frac{1}{2}(x_{t+1}^*-\beta_{t+1}^e)^\top (P_{t}^e-P_{t+1}^e)(x_{t+1}^*-\beta_{t+1}^e)\\
& \leq \frac{1}{2}\|x_{t+1}^*-\beta_{t+1}^e\|_2^2 \|P_t^e-P_{t+1}^e\|_2 \leq \frac{1}{2}(\bar x+\bar \beta )^2  \|P_t^e-P_{t+1}^e\|_2
\end{align*}

Therefore, we have
\begin{align}
\sum_{t=0}^{N}(h^e_{t-1}(x_{t}^*)-h^e_{t}(x_{t}^*))& \leq \sum_{t=0}^{N-1}(h_t^e(x_{t+1}^*)-h_{t+1}^e(x_{t+1}^*)) \notag\\
&=O( \sum_{t=0}^{N-1}( \| \beta_{t+1}^e-\beta_t^e\|_2+   \|P_t^e-P_{t+1}^e\|_2)) \label{equ: bdd hk(x)-hk+1(x)}
\end{align}
where the first inequality is by  $h_0^e(x) \geq 0$ and $h_{-1}^e(x)=0$. 
Thus, by  Theorem \ref{thm: initialization}, we have \small{
\begin{align*}
& \text{Regret}(RHTM) = O\left( \zeta^2(\frac{\sqrt{\zeta}-1}{\sqrt \zeta})^{2K}\left(\sum_{t=1}^N\left(\| P^e_t-P^e_{t-1}\| +  \|   \beta_t^e-\beta^e_{t-1}\|\right) + \sum_{t=0}^{N}\| x_{t-1}^e- x_t^e \|  \right)\right)
\end{align*}}
\end{proof}

\subsection{Proof of Corollary \ref{cor: LQT Q, R not change}}

Roughly speaking, the proof is mostly by applying Corollary \ref{cor: LQT regret upper bdd} and by showing $\|\beta_t^e-\beta_{t-1}^e\|$ and $\|x_t^e-x_{t-1}^e\|$ can be bounded by $\|\theta_t-\theta_{t-1}\|$ up to some constants and $\|P_t^e-P^e_{t-1}\|=0$ in the LQT problem \eqref{equ: LQT Q R not changing} where $Q$ and $R$ are not changing. However, directly applying  the results in Theorem \ref{thm: initialization} and Corollary \ref{cor: LQT regret upper bdd} will result in some extra constant terms because some inequalities used to derive  the bounds in Theorem \ref{thm: initialization} and Corollary \ref{cor: LQT regret upper bdd} are not necessary when $Q, R$ are not changing. Therefore, we will need some intermediate results in the proofs of Theorem \ref{thm: initialization} and Corollary \ref{cor: LQT regret upper bdd}  to prove Corollary \ref{cor: LQT Q, R not change}.

Firstly, by  Lemma \ref{lem: J(varphi)-sum lambda e} and Lemma \ref{lem: sum lambda e -J*}, we have
\begin{align*}
    J(FOSS)- J^* & = J(FOSS)-\sum_{t=0}^{N-1} \lambda_t^e + \sum_{t=0}^{N-1} \lambda_t^e-J^*\\
    & \leq  \underbrace{c_1 \sum_{t=0}^{N-1}\| x_{t-1}^e- x_t^e \|}_\text{Part I}+ \underbrace{\sum_{t=0}^{N-1}(h^e_t(x_{t+1}^*)-h^e_{t+1}(x_{t+1}^*))}_\text{Part II} +\underbrace{f_N(x_N(0)) -h^e_0(x_0)}_\text{Part III}
\end{align*}
We are going to bound each part by $\sum_t \|\theta_t -\theta_{t-1}\|$ in the following.

\nbf{Part I:} We will bound Part I  by $\sum_t \|\theta_t-\theta_{t-1}\|$ through showing that $x_t^e=F_1F_2\theta_t$ for some matrices $F_1, F_2$. The representation of $x_t^e$ relies on Lemma \ref{lem: x e x(0) form}.

By Lemma \ref{lem: x e x(0) form}, any steady state $(x,u)$ can be represented as a matrix multiplied by $z$: 
\begin{align*}
x&=(\underbrace{z^1, \dots, z^1}_{p_1}, \underbrace{z^2, \dots, z^2}_{p_2}, \dots, \underbrace{z^m, \dots, z^m}_{p_m})^\top \eqqcolon F_1z \\
u& = (z^1, \dots, z^m)^\top -A(\I, :) x  = (I_m-A(\I, :) F_1) z
\end{align*} 
where $F_1\in \R^{n,m}$ is a binary matrix with full column rank. 

Consider cost function $\frac{1}{2}(x-\theta)^\top Q(x-\theta)+ \frac{1}{2}u^\top  R u$. By the steady-state representation above, the optimal steady state can be solved by the following unconstrained optimization:
$$\min_z (F_1z-\theta)^\top Q (F_1z-\theta)+ z^\top  (I-A(\I, :) F_1)^\top R(I-A(\I, :) F_1) z$$
Since $F_1$ is full column rank, the function is strongly convex and has the unique solution 
\begin{equation}\label{equ: ze = F2 theta}
    z^e=F_2\theta
\end{equation}  
where $F_2 = (F_1^\top  Q F_1+ (I-A(\I, :) F_1)^\top R(I-A(\I, :) F_1))^{-1}F_1^\top Q$. Accordingly, the optimal steady state can be represented as  
\begin{equation}
    x^e= F_1 F_2\theta, \qquad u^e= (I_m-A(\I, :) F_1) F_2 \theta. \label{equ: xe, ue=Ftheta}
\end{equation} Consequently, when $1 \leq t \leq N-1$,$
\| x_t^e-x_{t-1}^e\| \leq \|F_1 F_2\| \| \theta_t -\theta_{t-1}\| $.
When $t=0$, $\|x_0^e -x_{-1}^e\| \leq\| F_1 F_2\| \| \theta_0 -\theta_{-1}\| $ holds since $x_{-1}^e= x_0=\theta_{-1}=0$.
Combining the upper bounds above, we have
\begin{align*}
    \text{Part I} =O\left( \sum_{t=0}^{N-1}\| \theta_t -\theta_{t-1}\| \right)
\end{align*}

\nbf{Part II:} 
By \eqref{equ: bdd hk(x)-hk+1(x)} in the proof of Corollary \ref{cor: LQT regret upper bdd}, and by noticing that $P_t^e=P^e(Q, R)$ does not change, we have
\begin{align*}
 \sum_{t=0}^{N-1}(h_t^e(x_{t+1}^*)-h_{t+1}^e(x_{t+1}^*))& = O\left(\sum_{t=0}^{N-1} \| \beta_{t+1}^e-\beta_t^e\|\right) 
\end{align*}
By Lemma \ref{lem: he formula}, $\beta_t^e = F(Q, R) \theta_t$  for $0\leq t \leq N-1$. In addition, since $\beta_N^e=\theta_N=0$ as defined in \eqref{equ: LQT Q R not changing} and Corollary \ref{cor: LQT regret upper bdd}, we can also write  $\beta_N^e=F(Q, R) \theta_N$.
Thus, 
\begin{equation*}
    \text{Part II}=O\left(\sum_{t=0}^{N-1}  \| \beta_{t+1}^e-\beta_t^e\|\right) =O\left(  \sum_{t=1}^{N} \|\theta_{t} -\theta_{t-1}\|\right)
\end{equation*}

\nbf{Part III:} 
By our condition for the terminal cost function, we have $f_N(x_N(0))= \frac{1}{2}(x_N(0)-\beta_N^e)^\top  P^e (x_N(0)-\beta_N^e)$. By Lemma \ref{lem: he formula}, we have $h_0^e(x_0)= \frac{1}{2}(x_0-\beta_0^e)^\top  P^e (x_0-\beta_0^e)$. So Part III can be bounded by 
\begin{align*}
    \text{Part III}& = \frac{1}{2}(x_N(0)-\beta_N^e)^\top  P^e (x_N(0)-\beta_N^e)-\frac{1}{2}(x_0-\beta_0^e)^\top  P^e (x_0-\beta_0^e)\\
    & = \frac{1}{2}(x_N(0)-\beta_N^e+ x_0-\beta_0^e)^\top  P^e (x_N(0)-\beta_N^e-(x_0-\beta_0^e)) \\
    & \leq  \frac{1}{2} \| x_N(0)-\beta_N^e+ x_0-\beta_0^e\| \| P^e\| \| x_N(0)-\beta_N^e-(x_0-\beta_0^e)\|\\
    &\leq \frac{1}{2}( \sqrt n \bar x^e+ \bar \beta  + \bar \beta)\| P^e\| (\| x_N(0)- x_0\| + \| \beta_N^e -\beta_0^e\|) 
\end{align*}
where the last inequality is by $x_0=0$, Lemma \ref{lem: bdd xte, ute, xt(0), ut(0)}, Lemma \ref{lem: uniform bdd on beta e}.

Next we will bound $\| x_N(0)- x_0\|$ and  $\| \beta_N^e -\beta_0^e\|$ respectively. Firstly,  by $\beta_t^e=F(Q,R)\theta_t$ in Lemma \ref{lem: he formula}, we have
\begin{align*}
    \| \beta_N^e -\beta_0^e\|& \leq \sum_{t=0}^{N-1} \| \beta_{t+1}^e-\beta_t^e\|\leq \| F(Q,R)\|\sum_{t=0}^{N-1} \| \theta_{t+1}-\theta_t\|
\end{align*}
Secondly, we will bound $\| x_N(0)- x_0\| $. 
\begin{align*}
\| x_N(0)- x_0\|  &\leq \| x_N(0)- x_{N-1}^e\| + \| x_{N-1}^e- x_0\|\\
& \leq \| x_N(0)- x_{N-1}^e\| + \sum_{t=0}^{N-1}\|x_t^e - x_{t-1}^e\|\\
&\leq \| x_N(0)- x_{N-1}^e\| +\| F_1 F_2\| \sum_{t=0}^{N-1}\|\theta_t - \theta_{t-1}\| 
\end{align*}
where the second inequality is by $x_0^e=x_0$, the third inequality is by \eqref{equ: xe, ue=Ftheta}.

Next, we will focus on $\| x_N(0)- x_{N-1}^e\|$. By Lemma \ref{lem: x e x(0) form}, 
\begin{align*}
    x_N(0)&= (z_{N-p_1}^{e,1}, \dots, z_{N-1}^{e,1}, z_{N-p_2}^{e,2}, \dots, z_{N-1}^{e,2}, \dots, z_{N-p_m}^{e,m}, \dots, z_{N-1}^{e,m})^\top \\
    x_{N-1}^e&= (z_{N-1}^{e,1}, \dots, z_{N-1}^{e,1}, z_{N-1}^{e,2}, \dots, z_{N-1}^{e,2}, \dots, z_{N-1}^{e,m}, \dots, z_{N-1}^{e,m})^\top
\end{align*}
As a result,
\begin{align*}
    \| x_N(0)-x_{N-1}^e\|^2& \leq \| z_{N-2}^e -z_{N-1}^e\|^2 + \dots + \| z_{N-p}^e - z_{N-1}^e\|^2\\
    & =\|F_2\|^2(\|\theta_{N-2}-\theta_{N-1}\|^2 + \dots + \|\theta_{N-p}-\theta_{N-1}\|^2)
\end{align*}
where the equality is by \eqref{equ: ze = F2 theta}.
Taking square root on both sides yields
\begin{align*}
    \| x_N(0)-x_{N-1}^e\|& \leq \|F_1\|\sqrt{\|\theta_{N-2}-\theta_{N-1}\|^2 + \dots + \|\theta_{N-p}-\theta_{N-1}\|^2}\\
    & \leq \|F_2\| (\|\theta_{N-2}-\theta_{N-1}\| + \dots + \|\theta_{N-p}-\theta_{N-1}\|)\\
    & \leq \|F_2\| (p-1) \sum_{t=N-p}^{N-2}\|\theta_{t+1}-\theta_t\|
\end{align*}
Combining the bounds above leads to
\begin{equation*}
    \text{Part III}=O\left(\sum_{t=0}^{N-1} \|\theta_{t+1} -\theta_{t}\|\right)
\end{equation*}
The proof is completed by summing up the bounds for Part I, II, III.

\section{Proof of Theorem \ref{thm: lower bdd}}\label{append: lower bound}

\nit{Proof intuition:}
By the problem transformation in Section \ref{sec: problem transform}, the fundamental limit of the online control problem is equivalent to the fundamental limit of the online convex optimization problem with objective $C(\mbf z)$. Therefore, we will focus on $C(\mbf z)$. Since the lower bound is for the worst case scenario, we only need to construct some tracking trajectories $\{\theta_t\}$ for Theorem \ref{thm: lower bdd} to hold.   However, it is generally difficult to construct the tracking trajectories, so we consider randomly generated $\theta_t$ and show that the regret in expectation can be lower bounded. Then, there must exist some realization of the randomly generated $\{\theta_t\}$ such that the regret lower bound holds. 


\nit{Formal proof:}

\nbf{Step 1: construct LQ tracking.} For simplicity, we construct a single-input system with $n=p$ and  $A \in \R^{n,n}$ and $B\in \R^{n\times 1}$ as follows: \footnote{It is easy to generalize the construction to multi-input case by constructing $m$ decoupled subsystems. }
\begin{align*}
A= \left(\begin{array}{ccccc}
0 & 1 & \cdots   &0  \\
\vdots & \ddots & \ddots  &\\
& &  0 & 1\\
1 &0 & \cdots  & 0
\end{array}\right), \quad B =\left(
\begin{array}{c}
0 \\
\vdots \\
0 \\
1
\end{array}
\right)
\end{align*}
$(A, B)$ is controllable because $(B, AB, \dots, A^{p-1}B)$ is full rank. $A$'s controllability index is $p=n$. 

Next, we construct $Q$ and $R $. For any $\zeta>1$ and $p$, define $\delta = \frac{4}{(\zeta-1)p}$. Let $Q= \delta I_n$  and $R= 1$ for $0 \leq t \leq N-1$. Let $P^e=P^e(Q, R)$ be the solution to the DARE. The next lemma shows that $P^e$ is a diagonal matrix and its diagonal entries can be characterized.
\begin{lemma}[Form of $P^e$] \label{lem: Pe's form}
    Let $P^e$ denote the solution to the DARE determined by $A, B, Q, R$ defined above. Then $P^e$ satisfies the form
    \begin{align*}
P^e= \left(\begin{array}{ccccc}
q_1 &0 & \cdots &0 \\
0& q_2 & \cdots &0 \\
& & \ddots  &\\
0&&\cdots & q_n
\end{array}\right),
\end{align*}
where $q_i = q_1+(i-1)\delta$ for $1 \leq i \leq n$ and $\delta < q_1 < \delta +1$.
\end{lemma}
\begin{proof}[Proof of Lemma \ref{lem: Pe's form}]
The DARE exists a unique positive definite solution \cite{bertsekas3rddynamic}.
Suppose the solution is diagonal and substitute it in the DARE as follows.
 \begin{align*}
 \qquad\qquad\quad  P^e
&= Q+A^\top (P^e-P^eB(B^\top P^eB+R)^{-1}B^\top P^e)A \\
      \left(\begin{array}{ccccc}
q_1 &0 & \cdots &0 \\
0& q_2 & \cdots &0 \\
& & \ddots  &\\
0&&\cdots & q_n
\end{array}\right)
       &= \left(\begin{array}{ccccc}
q_n/(1+q_n)+\delta &0 & \cdots &0 \\
0& q_1+\delta & \cdots &0 \\
& & \ddots  &\\
0&&\cdots & q_{n-1}+\delta
\end{array}\right)
    \end{align*}
    So we have $q_i = q_{i-1}+\delta$ for $1 \leq i \leq n-1$, and $q_n/(1+q_n)+\delta=q_1= q_n-(n-1)\delta$. Thus, $q_n = \frac{n\delta + \sqrt{n^2\delta^2 + 4n \delta}}{2} >n\delta$. It is straightforward that $q_1= q_n-(n-1)\delta >\delta >0$, and $q_1< \delta +1$ by  $q_n/(1+q_n)<1$. So we have found the unique positive definite solution to the DARE.
\end{proof}

Next, we will construct $\theta_t$. Let $\theta_0=\theta_N=\beta_N^e=0$ for simplicity. For $\theta_t$ when $1\leq t \leq N-1$, we divide the $N-1$ stages into $E$ epochs, each with length $\Delta = \ceil{\frac{N-1}{\floor{\frac{L_N}{2\bar\theta}}}}$, possibly except  the last epoch. This is possible because $1\leq \Delta \leq N-1$ by the conditions in Theorem \ref{thm: lower bdd}. Thus, $E= \ceil{\frac{N-1}{\Delta}}$. 
Let $\J$ be the first stage of the each epoch: $\J= \{1, \Delta+1, \dots, (E-1)\Delta +1 \}$. Let $\theta_t$ for $t\in \J$ independently and identically follow the  distribution below. \begin{align*}
\Pr(\theta_t^i=a) &=\begin{cases}
1/2 & \text{if } a=\sigma\\
1/2 & \text{if } a=-\sigma
\end{cases}, \quad \text{i.i.d. for all $i \in[n]$, $t\in \J$,}
\end{align*}
where $\sigma= \frac{\bar \theta}{\sqrt{n}}$. It can be easily verified that $\|\theta\|= \bar \theta$ for any realization of this distribution, so Assumption \ref{ass: theta xi bdd} is satisfied. Let the other  $\theta_t$ in each epoch be equal to the $\theta$ at the start of their corresponding epochs, i.e. $ \theta_{k\Delta +1}=\theta_{k\Delta +2}=\cdots =\theta_{(k+1)\Delta }$, when $k \leq E-1$, and $\theta_{k\Delta +1}=\dots =\theta_{N-1}$ when $k=E$. The following inequalities show that the constructed $\{\theta_t\}$ satisfies the variation budget:
\begin{align*}
     \sum_{t=0}^N \| \theta_t -\theta_{t-1}\|&=\|\theta_1-\theta_0\| + \sum_{k=1}^{E-1}\|\theta_{k\Delta+1}-\theta_{k\Delta}\| + \| \theta_{N-1}-\theta_N\|\\
     & \leq \bar \theta + 2(E-1)\bar \theta + \bar \theta=2\bar \theta E\\
     & \leq 2\bar \theta\floor{\frac{L_N}{2\bar\theta}} \leq 2\bar \theta\frac{L_N}{2\bar\theta}= L_N
\end{align*}
where the first equality is 
by $\theta_0=\theta_{-1}=\theta_N=0$, the first inequality is by $\|\theta_t \|=\bar \theta$ when $1 \leq t \leq N-1$, the second inequality is by $\Delta = \ceil{\frac{N-1}{\floor{\frac{L_N}{2\bar\theta}}}}\geq \frac{N-1}{\floor{\frac{L_N}{2\bar\theta}}}$,   and thus $ \floor{\frac{L_N}{2\bar\theta}} \geq \ceil{\frac{N-1}{\Delta} }=E$.

The  total cost of our constructed LQ tracking problem is
$$J(\mbf x,\mbf u)= \sum_{t=0}^{N-1}(\frac{\delta}{2}\|x_t-\theta_t\|^2 + \frac{1}{2}u_t^2)+ \frac{1}{2}x_N^\top  P^e x_N$$

We will verify that $C(\mbf z)$'s condition number is $\zeta$ in Step 2.

\nbf{Step 2: problem transformation and  the optimal solution $\mbf z^*$.} By the problem transformation in Section \ref{sec: problem transform}, we let $z_t =x_t^n$, and the equivalent cost function $C(\mbf z)$ is given below.
\begin{align*}
& C(\mbf z)=\sum_{t=0}^{N-1}(\frac{\delta}{2}\sum_{i=1}^n ( z_{t-n+i}- \theta_t^i)^2 +\frac{1}{2}(z_{t+1}- z_{t-n+1})^2) +  \frac{1}{2}\sum_{i=1}^n q_iz_{N-n+i}^2
\end{align*}
and $z_t =0$ and $\theta_t=0$ for $t \leq 0$. 

Since $C(\mbf z)$ is strongly convex, $\min C(\mbf z)$ admits a unique optimal solution, denoted as $\mbf z^*$, which is determined by the first-order optimality condition: $\nabla C(\mbf z^*)=0$. In addition, our constructed $C(\mbf z)$ is a quadratic function, so there exists a matrix $H \in \R^{N\times N}$ and a vector $\eta\in \R^N$ such that $\nabla C(\mbf z^*)=H\mbf z^*-\eta=0$. By the partial gradients of $C(\mbf z)$ below,
\begin{align*}
& \frac{\partial C}{\partial z_t}= \delta ( z_t - \theta_t^n + z_t -\theta_{t+1}^{n-1} +\dots + z_t - \theta_{t+n-1}^1) + z_t - z_{t+n}+z_t - z_{t-n}, \ 1 \leq t \leq N-n\\
&  \frac{\partial C}{\partial z_t}= \delta ( z_t - \theta_t^n +\dots + z_t - \theta_{N-1}^{n+t-N+1} ) + q_{n+t-N}z_t+z_t-z_{t-n}, \qquad N-n+1 \leq t \leq N 
\end{align*}
For simplicity and without loss of generality, we assume  that $N/n$ is an integer. Then, by Lemma \ref{lem: Pe's form},
 $H$ can be represented as  the block matrix below 
$$
H = \left(
\begin{array}{cccc}
(\delta n+2)I_n& -I_n & \cdots &\\
-I_n & (\delta n+2)I_n & \ddots & \\
&  \ddots & \ddots & -I_n\\
& & -I_n & (q_n+1)I_n
\end{array}\right)\in \R^{N \times N}.
$$
 $\eta$ is a linear combination of $\theta$: for $1 \leq t \leq N$, we have 
$\eta_t = \delta(  \theta_t^n+\dots +   \theta_{t+n-1}^1) = \delta(e_n^\top   \theta_t  + \dots + e_1^\top   \theta_{t+n-1})$ 
where $e_1, \dots, e_n\in \R^n$ are standard basis vectors and $\theta_t=0$ for $t\geq N$.


By Gergoskin's Disc Theorem and Lemma \ref{lem: Pe's form}, $H$'s condition number is $(\delta n+4)/\delta n =\zeta$ by our choice of $\delta$ in Step 1 and $p=n$. Thus we have shown that $C(\mbf z)$'s condition number is $\zeta$.

Since $H$ is strictly diagonally dominant with positive diagonal entries and nonpositive off-diagonal entries, $H$ is invertible and its inverse, denoted by $Y$, is nonnegative. Consequently,  the optimal solution can be represented as $\mbf z^*=Y\eta$. Since $\eta$ is linear in $\{ \theta_t \}$, $z_t^*$ is also linear in $\{ \theta_t \}$ and can be characterized by the following.
\begin{align}
z_{t+1}^*& = \sum_{i=1}^N Y_{t+1,i} \eta_{i} \nonumber =\delta\sum_{i=1}^N Y_{t+1,i} \sum_{j=0}^{n-1} e_{n-j}^\top    \theta_{i+j}  \nonumber\\
& = \delta \sum_{k=1}^{N-1} \left( \sum_{i=1}^n Y_{t+1,i+k-n}e_i^\top \right) \theta_k\nonumber \\
& \eqqcolon \delta \sum_{k=1}^{N-1} v_{t+1,k}   \theta_k\label{equ: z represented by theta bar}
\end{align}
where  $\theta_t=0$ for $t\geq N$, $Y_{t+1,i}=0$ for $i\leq 0$, and $v_{t+1,k}\coloneqq \sum_{i=1}^n Y_{t+1,i+k-n}e_i^\top$.

In addition, we are able to show in the next lemma that $Y$ has decaying row entries starting at the diagonal entries. The proof is technical and deferred to the Appendix F.1. 
\begin{lemma}\label{lem: H and H inverse properties}
When $N/n$ is an integer, 
the inverse of $H$, denoted by $Y$, can be represented as a block matrix
$$
Y = \left(
\begin{array}{cccc}
y_{1,1} I_n& y_{1,2} I_n & \cdots &y_{1,N/n} I_n\\
y_{2,1} I_n & y_{2,2} I_n & \cdots & y_{2,N/n} I_n\\
\vdots&  \ddots & \ddots & \vdots\\
y_{N/n,1} I_n&y_{N/n,2}I_n & \cdots  & y_{N/n,N/n} I_n
\end{array}\right)
$$
where $y_{t, t+\tau} \geq \frac{1-\rho}{\delta n +2}\rho^\tau>0$ for $\tau \geq 0$ and $\rho= \frac{\sqrt \zeta -1}{\sqrt \zeta +1}$.
\end{lemma}

\nbf{Step 3: characterize $z_{t+1}(\A^z)$.}
For any online control algorithm $\A$, we can define an equivalent online algorithm for $z$, denoted as $\A^z$. $\A^z$, at each time $t$, outputs $z_{t+1}(\A^z)$  based on the predictions and the history, i.e.,
$$z_{t+1}(\A^z) = \A^z( \{ \theta_s \}_{s=0}^{t+W-1}), \quad t\geq 0$$
For simplicity, we consider online deterministic algorithm.\footnote{The proof can be easily generalized to random algorithms} Notice that $z_{t+1}$ is a random variable because $\theta_1, \dots, \theta_{t+W-1}$ are random. Based on this observation and Lemma \ref{lem: H and H inverse properties}, we are able to provide a regret lower bound in Step 4.

\nbf{Step 4: prove the regret lower bound on $\A$.}
Roughly speaking, the regret occurs when something unexpected happens beyond the prediction window, that is,  at each $t$, the prediction window goes as far as $t+W-1$, but if $\theta_{t+W}$ changes from $\theta_{t+W-1}$, the online algorithm cannot prepare for it, resulting in poor control and positive regret. 
By our construction, when $t+W\in \J$, $\theta_{t+W}$ changes from $\theta_{t+W-1}$. To study such $t$, we  define a set  $\J_1 = \{ 0 \leq t \leq N-W-1\mid t+W \in \J\}$.
It can be shown that the cardinality of $\J_1$ can be lower bounded by $L_N$ up to some constants:
\begin{align}\label{equ: |J1| lower bdd}
    |\J_1| \geq \frac{1}{18 \bar \theta}L_N
\end{align}
The proof of \eqref{equ: |J1| lower bdd} is provided below.
\begin{align*}
    |\J_1| &= | \{ W\leq t \leq N-1\mid t\in \J\}|\\
    &=| \J|-| \{ 1\leq t \leq W-1\mid t\in \J\}|\\
    & = \ceil{\frac{N-1}{\Delta}}- \ceil{\frac{W-1}{\Delta}}\\
    &\geq \floor{\frac{N-W}{\Delta}} \\
    &\geq \frac{1}{2}\frac{N-W}{\Delta}\\
    &\geq  \frac{1}{2}\frac{N-W}{N-1 + \floor{\frac{L_N}{2\bar \theta}}}  \floor{\frac{L_N}{2\bar \theta}} \\
    &\geq \frac{1}{2} \frac{N-\frac{1}{3}N}{N-1+N+1/2}\floor{\frac{L_N}{2\bar \theta}}\geq \frac{1}{6}\floor{\frac{L_N}{2\bar \theta}}\\
    &\geq \frac{1}{6} \frac{2}{3}\frac{L_N}{2\bar \theta}= \frac{1}{18}\frac{L_N}{\bar \theta}
\end{align*}
where the first inequality is by the definition of the ceiling and floor operators, the second inequality is by $\frac{N-W}{\Delta} \geq 1$ under the conditions on $N, W,  L_N$ in Theorem \ref{thm: lower bdd}, the third inequality is by $\Delta = \ceil{\frac{N-1}{\floor{\frac{L_N}{2\bar\theta}}}} \leq \frac{N-1}{\floor{\frac{L_N}{2\bar\theta}}}+1$, the fourth inequality is by $L_N \leq (2N+1)\bar \theta$ in Theorem 3's statement, the last inequality is by $L_N \geq 4 \bar \theta$ in Theorem 3's statement.

Moreover, we can show  in Lemma \ref{lem: bdd E(z A- zopt)2} that, for all  $t \in \J_1$,  the online decision $z_{t+1}(\A^z)$ is different from the optimal solution  $z_{t+1}^*$ and the difference is lower bounded,   

\begin{lemma} \label{lem: bdd E(z A- zopt)2}For any online algorithm $\A^z$,
    when $t \in \J_1$, 
    $$\E |z_{t+1}(\A^z)-z_{t+1}^*|^2 \geq c_{10}\sigma^2\rho^{2K}$$
    where $c_{10}$ is a constant determined by $A, B, n, Q, R$ constructed above and $\rho= \frac{\sqrt \zeta -1}{\sqrt \zeta +1}$.
\end{lemma}
The proof is provided in Appendix F.2.

The lower bound on the difference between the online decision and the optimal decision results in a lower bound on the regret. 
By the $n\delta$-strong convexity of $C(\mbf z)$,
\begin{align*}
\E (C(\mbf z(\A^z))-C(\mbf z^*) )&\geq \frac{\delta n}{2}\sum_{t\in \J_1}\E |z_{t+1}(\A^z)-z_{t+1}^*|^2\\
&\geq |\J_1| c_{10}\sigma^2\rho^{2K}\\
&\geq \frac{L_N}{18 \bar \theta} c_{10}\sigma^2\rho^{2K} = \Omega(L_N \rho^{2K})
\end{align*}

By the equivalence between $\A$ and $\A^z$, we have
$\E J(\A)-\E J^* = \Omega(\rho^{2K}L_N)$. 
By the property of expectation, there must exist some realization of the random $\{\theta_t\}$ such that $ J(\A)- J^* = \Omega(\rho^{2K}L_N)$, where $\rho= \frac{\sqrt \zeta -1}{\sqrt \zeta +1}$. This completes the proof. \qed

\subsection{Proof of Lemma \ref{lem: H and H inverse properties}}

\begin{proof}
	
	Since $H$ is a block matrix
	$$
	H = \left(
	\begin{array}{cccc}
	(\delta n+2)I_n& -I_n & \cdots &\\
	-I_n & (\delta n+2)I_n & \ddots & \\
	&  \ddots & \ddots & -I_n\\
	& & -I_n & (q_n+1)I_n
	\end{array}\right)
	$$
	its inverse matrix $Y$ can also be represented as a block matrix. Moreover, let 
	\[
	H_1 = \left(\begin{array}{cccc}
	\delta n +2 & -1 & \cdots & 0\\
	-1 & \delta n +2 & \ddots & 0\\
	\vdots & \ddots & \ddots & \vdots\\
	0 & \cdots & -1 & q_n+1
	\end{array}\right)
	\] and define $\bar Y=(H_1)^{-1}=(y_{ij})_{i,j =1}^{N/n}$. Then the inverse matrix $Y$ can be represented as the block matrix: $Y=(y_{ij}I_n)_{i,j =1}^{N/n}$. 
	
	Now, it suffices to provide a lower bound on $y_{ij}$.


	Since $H_1$ is a symmetric positive definite tridiagonal matrix, by \cite{concus1985block}, the inverse has an explicit formula given by $(H_1)^{-1}_{ij}= a_ib_j$ and 
	\begin{align*}
	a_t &= \frac{\rho}{1- \rho^2}\left( \frac{1}{\rho^t}   -\rho^t\right)\\
	b_t &= c_3 \frac{1}{\rho^{N-t}} + c_4 \rho^{N-t} \\
	c_3 & = b_N\left( \frac{(q_n+1)\rho -\rho^2}{1-\rho^2}  \right)\\
	c_4 & = b_N \frac{1-(q_n+1)\rho}{1-\rho^2}\\
	b_N& = \frac{1}{- a_{N-1}+ (q_n+1)  a_N}
	\end{align*}
	
	In the following, we will show $y_{t,t+\tau}= a_t b_{t+\tau} \geq \frac{1-\rho}{\delta n +2} \rho^\tau$ when $\tau\geq 0$.  Firstly, it is easy to verify that
	$$\rho^t   a_t = \frac{\rho}{1-\rho^2}(1-\rho^{2t})\geq \rho$$
	since $t\geq 1$ and $\rho<1$.
	
	Secondly, we bound $b_N$ in the following way:
	\begin{align*}
	\rho^{-N} b_N  = \frac{1}{(q_n+1)(1-\rho^{2N})-(\rho-\rho^{2N-1}) } \frac{1-\rho^2}{\rho}\geq \frac{1}{(\delta n+2) } \frac{1-\rho^2}{\rho}  
	\end{align*}
	because $0<(q_n+1)(1-\rho^{2N})-(\rho-\rho^{2N-1}) \leq (\delta n+2)$ by $n\delta <q_n <n\delta +1$ in Lemma \ref{lem: Pe's form}.

	Thirdly, we bound $b_{t+\tau}$. When $1-(q_n+1)\rho \geq 0$
	\begin{align*}
	\rho^{N-t-\tau}b_{t+\tau}&=  b_N\left( \frac{(q_n+1)\rho -\rho^2}{1-\rho^2}  \right) +  b_N \frac{1-(q_n+1)\rho}{1-\rho^2} \rho^{2(N-t-\tau)}\\
	& \geq b_N\left( \frac{(q_n+1)\rho -\rho^2}{1-\rho^2}  \right) \\
	& \geq b_N\left( \frac{(\delta n+1)\rho -\rho^2}{1-\rho^2}  \right) \\
	& = \frac{1-\rho}{1-\rho^2}b_N
	\end{align*}
	where the first inequality is by
	 $1-(q_n+1)\rho \geq 0$, the second inequality is by $qn >n\delta$ in Lemma \ref{lem: Pe's form}, and
	the last equality is by $\rho^2 - (\delta n+2)\rho +1=0$.
	
	When  $1-(q_n+1)\rho < 0$
	\begin{align*}
	\rho^{N-t-\tau}b_{t+\tau}&=  b_N\left( \frac{(q_n+1)\rho -\rho^2}{1-\rho^2}  \right) +  b_N \frac{1-(q_n+1)\rho}{1-\rho^2} \rho^{2(N-t-\tau)}\\
	& \geq b_N\left( \frac{(q_n+1)\rho -\rho^2}{1-\rho^2}  \right) +  b_N \frac{1-(q_n+1)\rho}{1-\rho^2} \\
	& \geq b_N \geq \frac{1-\rho}{1-\rho^2}b_N
	\end{align*}
	where the first inequality is by $1-(q_n+1)\rho < 0, \rho\leq 1$, the second inequality is by $ \rho^{2(N-t-\tau)}\leq 1$. Thus, we obtained a lower bound for $b_{t+\tau}$.
	
	Combining bounds of $a_t, b_{t+\tau}, b_N$ together yields
	\begin{align*}
	y_{t,t+\tau}= a_tb_{t+\tau}&\geq \rho b_N\frac{1-\rho}{1-\rho^2} \rho^{\tau-N}\geq \frac{1-\rho}{(\delta n+2)}\rho^\tau
	\end{align*}
	
\end{proof}

\subsection{Proof of Lemma \ref{lem: bdd E(z A- zopt)2}}

\begin{proof}
    By our construction, $\theta_t$ is random,  $z_{t+1}^\A$ is also random and its randomness is provided by $\theta_1, \dots, \theta_{t+W-1}$, while $z_{t+1}^* $ is determined by all $\theta_t$. When $t\in \J_1$, 
	\begin{align*}
	\E |z_{t+1}^\A-z^*_{t+1}|^2 &= \E | z_{t+1}^\A- \delta\sum_{i=1}^{N-1} v_{t+1,i}  \theta_i|^2  \\
	& = \E |z_{t+1}^\A- \delta\sum_{i=1}^{t+W-1}v_{t+1,i}  \theta_i \|^2 +\delta^2 \E | \sum_{i=t+W}^{N-1} v_{t+1,i}  \theta_i|^2 \\
	& \geq \delta^2 \E | \sum_{i=t+W}^{N-1} v_{t+1,i}  \theta_i|^2,
	\end{align*}
	where the first equality is by \eqref{equ: z represented by theta bar}, the second equality is by $\E\theta_\tau=0$ for all $\tau$, and $\theta_{t+W}, \dots, \theta_N$ are independent of $\theta_1, \dots, \theta_{t+W-1}$ when $t\in \J_1$.
	
	Further,
	\begin{align*}
	     \E \large| \sum_{i=t+W}^{N-1} v_{t+1,i}  \theta_i|^2& = \E | \sum_{i=t+W}^{t+W+\Delta-1} v_{t+1,i}  \theta_{t+W}|^2+\dots+\E | \sum_{i=(E-1)\Delta+1}^{N-1} v_{t+1,i}  \theta_{(E-1)\Delta+1}|^2\\
	    & = \| \sum_{i=t+W}^{t+W+\Delta-1} v_{t+1,i}\|^2 \sigma^2+\dots + \| \sum_{i=(E-1)\Delta+1}^{N-1} v_{t+1,i}\|^2 \sigma^2\\
	    & \geq \sigma^2 \sum_{i=t+W}^{N-1}  \|v_{t+1,i}\|^2  \\
	    &= \sigma^2 \sum_{i=t+W}^{N-1} (\sum_{k=0}^{n-1} Y_{t+1,i-k}^2)  \geq \sigma^2 \sum_{i=t+1+W-n}^{N-1} Y_{t+1,i}^2  \\
	    &=\sigma^2 \sum_{i=t+1+W-n}^{N} Y_{t+1,i}^2
	\end{align*}
	where the first equality is because the theta in one epoch are equal by our construction,
	the second equality is because $\text{cov}(\theta_\tau)=\sigma^2 I_n$, 
	the first inequality is because  the entries of $v_{t+1,i}$  are nonnegative, the third equality is by the definition of $v_{t+1,i}$ in \eqref{equ: z represented by theta bar}, and
	the last equality is because when $t \in \J_1$, $Y_{t+1, N}=0$.
	
		When $1\leq W\leq n$, $\sum_{i=t+1+W-n}^{N} Y_{t+1,i}^2   \geq Y_{t+1,t+1}^2= Y_{t+1, t+1+n \floor{\frac{W-1}{n}}}$. When $W>n$, $\sum_{i=t+1+W-n}^{N} Y_{t+1,i}^2  \geq Y_{t+1,t+1+n\ceil{\frac{W-n}{n}}}^2 $. Moreover, when $W\geq 1$, $\ceil{\frac{W-n}{n}}=\floor{\frac{W-1}{n}}$. In summary, for  $W\geq 1$, 
	\begin{align*}
	&  \sum_{i=t+1+W-n}^{N} Y_{t+1,i}^2   \geq  Y_{t+1,t+1+n\floor{\frac{W-1}{n}}}^2 \geq \rho^{2 K}(\frac{1-\rho}{\delta n +2})^2
	\end{align*}
	where the last inequality is by Lemma \ref{lem: H and H inverse properties}. This completes the proof. 

\end{proof}

\section{Proofs of the LQT's properties used in Appendix \ref{append: LQT}}\label{append: LQT proof}

In this section, we provide proofs for the properties of LQ tracking (LQT) used in Appendix \ref{append: LQT}.

\subsection{Preliminaries: dynamic programming for finite-horizon LQT}\label{subsec: optimal soln to finite LQT}

In this section, we consider a discrete time LQ tracking problem with time-varying cost functions and time-invariant dynamical system:
\begin{align*}
\min_{x_t,u_t}& \ \ \frac{1}{2} \sum_{t=0}^{N-1} \left[ (x_t-\theta_t)^\top  Q_t (x_t-\theta_t)+ u_t^\top  R_t u_t \right] + \frac{1}{2} (x_{N}-\theta_{N})^\top  Q_{N}(x_{N}-\theta_{N})\\
\text{s.t. } \ & \ x_{t+1}=Ax_t + Bu_t, \quad \quad t=0, \dots, N-1
\end{align*}
where $x_0=0$ for simplicity. 

The problem can be solved by dynamic programming.

\begin{theorem}[Dynamic programming for  the finite-horizon LQT]\label{thm: DP solution to LQT}
	Consider a finite-horizon time-varying LQ tracking problem. Let $V_t(x_t)$ be the cost to go from $k=t$ to $k=N$, then 
	\begin{align*}
	V_t(x_t)= \frac{1}{2} (x_{t}-\beta_{t})^\top  P_{t}(x_{t}-\beta_{t})+\frac{1}{2} \sum_{k=t}^{N-1}  (A\theta_k  - \beta_{k+1})^\top  H_k (A\theta_k  - \beta_{k+1})
	\end{align*}
	for $t=0, \dots, N$. The parameters can be obtained by
	\begin{align*}
	 P_t& = Q_t + A^\top  M_t A, \quad t=0, \dots, N-1, \quad P_N=Q_N\\
	 M_t& = P_{t+1}-P_{t+1}B(R_t+B^\top  P_{t+1}B)^{-1}B^{T}P_{t+1}, \quad t=0, \dots, N-1\\
	\beta_t&  = (Q_t + A^\top  M_t A)^{-1} (Q_t \theta_t + A^\top  M_t \beta_{t+1}), \quad t=0, \dots, N-1\\
	\beta_{N}& = \theta_{N}\\
H_t	&  = M_t - M_t A(Q_t +A^\top M_t A)^{-1} A^\top  M_t , \quad t=0, \dots, N-1
	\end{align*}
	
	The optimal controller is 
	\begin{align*}
	u_t^* = - K_t x_t + K_t' \beta_{t+1}, \quad t=0, \dots, N-1 
	\end{align*}
	where the parameters are 
	\begin{align*}
	&K_t=(R_t + B^\top  P_{t+1}B)^{-1} B^\top  P_{t+1}A\\
	&K_t'=(R_t + B^\top  P_{t+1}B)^{-1} B^\top  P_{t+1}
	\end{align*}
	
	There is another way to write the optimal controller: 
	\begin{align*}
	u_t^* =- K_t x_t + K_t^{\alpha} \alpha_{t+1}  \quad t=0, \dots, N-1 
	\end{align*}
	where the parameters are 
	\begin{align*}
	 K_t^\alpha &= (R_t + B^\top  P_{t+1}B)^{-1} B^\top \\
	 \alpha_t &= P_t \beta_t\\
	 \alpha_t& = Q_t \theta_t +  (A-BK_t)^\top \alpha_{t+1}, \quad t=0, \dots, N-1\\
	 \alpha_{N}&=P_{N}\theta_{N}
	\end{align*}

\end{theorem}
\begin{proof}
    The proof is straightforward by following dynamic programming procedures.
    
    Firstly, it is direct to verify that $V_N(x_N)= \frac{1}{2}(x_N-\theta_N)^\top Q_N(x_N-\theta_N)$.
    Then, suppose the claim of Theorem \ref{thm: DP solution to LQT} is true at $t+1$, we will verify the stage $t$ in the following.
    \begin{align*}
        V_t(x_t)=&\min_{u_t}\big[\frac{1}{2}(x_t-\theta_t)^\top Q_t (x_t-\theta_t) +\frac{1}{2} u_t^\top R_t u_t+ V_{t+1}(Ax_t+Bu_t)\big]\\
         =&\frac{1}{2}\min_{u_t}\big[(x_t-\theta_t)^\top Q_t (x_t-\theta_t) + u_t^\top R_t u_t + (Ax_{t}+Bu_t-\beta_{t+1})^\top P_{t+1}(Ax_t+Bu_t -\beta_{t+1})\\
        &  +  \sum_{k=t+1}^{N-1}  (A\theta_k  - \beta_{k+1})^\top  H_k (A\theta_k  - \beta_{k+1})
        \big]\\
         = &\frac{1}{2}(Ax_t-\beta_{t+1})^\top (P_{t+1}-P_{t+1}B(R_t +B^\top P_{t+1}B)^{-1}B^\top P_{t+1})(Ax_t-\beta_{t+1})\\
        &  + \frac{1}{2}(x_t-\theta_t)^\top Q_t (x_t-\theta_t)+ \frac{1}{2}\sum_{k=t+1}^{N-1}  (A\theta_k  - \beta_{k+1})^\top  H_k (A\theta_k  - \beta_{k+1})\\
        =& \frac{1}{2}(Ax_t-\beta_{t+1})^\top M_t (Ax_t-\beta_{t+1})  + \frac{1}{2}(x_t-\theta_t)^\top Q_t (x_t-\theta_t)\\
        &+ \frac{1}{2}\sum_{k=t+1}^{N-1}  (A\theta_k  - \beta_{k+1})^\top  H_k (A\theta_k  - \beta_{k+1})\\
        =& \frac{1}{2}(x_t-\beta_t)^\top P_t(x_t-\beta_t) - \frac{1}{2}(Q_t \theta_t+ A^\top M_t \beta_{t+1})^\top(Q_t +A^\top M_t A)^{-1}(Q_t \theta_t+ A^\top M_t \beta_{t+1})\\
        &+ \frac{1}{2}\theta_t^\top Q_t \theta_t + \frac{1}{2} \beta_{t+1}^\top M_t \beta_{t+1}+
        \frac{1}{2}\sum_{k=t+1}^{N-1}  (A\theta_k  - \beta_{k+1})^\top  H_k (A\theta_k  - \beta_{k+1})\\
        =& \frac{1}{2}(x_t-\beta_t)^\top P_t(x_t-\beta_t)+
        \frac{1}{2}\sum_{k=t}^{N-1}  (A\theta_k  - \beta_{k+1})^\top  H_k (A\theta_k  - \beta_{k+1})
    \end{align*}
    where the third equality is by noticing that the optimal control input is
    $$u_t^*= -(R_t +B^\top P_{t+1}B)^{-1}B^\top P_{t+1}(Ax_t-\beta_{t+1})=-K_tx_t+K_t'\beta_{t+1},$$
    the fourth equality is by $M_t$'s definition, the fifth equality is by combining the two quadratic terms of $x_t$ as one quadratic term with a constant, and the last equality is by definition.
\end{proof}

\subsection{Proof of Lemma \ref{lem: uniform bdd on Pe}}
 In the following, we first prove that the recursive solution $P_t$ to the finite-horizon LQT is bounded. Then,  we can prove Lemma \ref{lem: uniform bdd on Pe} by taking limits.
\begin{lemma}[Bounded $P_t$ for finite-horizon LQT]\label{lem: Pt in Pbar Punderline}
Consider a finite-horizon time-varying LQT problem.	For any $N$, any $0\leq t \leq N$, any $Q_t \in \Q, R_t \in \Rc, Q_N \in \Pc$, we have $P_t \in \Pc$ where $P_t$ is defined in Theorem \ref{thm: DP solution to LQT}.
\end{lemma}
\begin{proof}
In the following, we use the notations and definitions introduced in Appendix \ref{append: LQT}.1 and Theorem \ref{thm: DP solution to LQT}.

Since $P_t$ does not depend on $\theta_t$, we let $\theta_t=0$ and consider the LQR problem for simplicity.
	Since $\underline Q\leq Q_t \leq \bar Q, \underline R\leq R_t \leq \bar R,
	$ for $0\leq t \leq N-1$ and $\underline P\leq Q_N \leq \bar P$, we have for any $x_t, u_t$, $k$, $Q_t, R_t, Q_N$,
	\begin{align*}
	& \sum_{t=k}^{N-1} (x_t^\top  Q_t x_t +u_t^\top  R_t u_t)+ x_N^\top Q_Nx_N \leq \sum_{t=k}^{N-1} (x_t^\top  \bar Q x_t +u_t^\top  \bar R u_t)+ x_N^\top \bar Px_N\\
	& \sum_{t=k}^{N-1} (x_t^\top  Q_t x_t +u_t^\top  R_t u_t)+ x_N^\top Q_Nx_N \geq \sum_{t=k}^{N-1} (x_t^\top  \underline Q x_t +u_t^\top  \underline R u_t)+ x_N^\top \underline Px_N
	\end{align*}
	Taking minimum over all feasible trajectories on both sides yields
	\begin{align*}
	& \min \sum_{t=k}^{N-1} (x_t^\top  Q_t x_t +u_t^\top  R_t u_t)+ x_N^\top Q_Nx_N \leq  \min\sum_{t=k}^{N-1} (x_t^\top  \bar Q x_t +u_t^\top  \bar R u_t)+ x_N^\top \bar Px_N\\
	&  \min\sum_{t=k}^{N-1} (x_t^\top  Q_t x_t +u_t^\top  R_t u_t)+ x_N^\top Q_Nx_N \geq  \min\sum_{t=k}^{N-1} (x_t^\top  \underline Q x_t +u_t^\top  \underline R u_t)+ x_N^\top \underline Px_N
	\end{align*}
	Notice that the left-hand-side terms of both inequalities are equal to $x_k^\top P_k x_k$. Moreover, notice that
	\begin{align*}
	x_k^\top \bar P x_k=\min_{x_{t+1}=Ax_t+Bu_t}\sum_{t=k}^{N-1} (x_t^\top  \bar Q x_t +u_t^\top  \bar R u_t)+ x_N^\top \bar Px_N
	\end{align*}
	because $\bar P=P^e(\bar Q, \bar R)$ is the solution to the DARE. The same holds for $\underline P$.
	Therefore, we have
	$$ x_k^\top \underline P x_k \leq x_k^\top  P_k x_k\leq x_k^\top \bar P x_k$$
	for any $x_k$. Thus, $\underline P \leq P_k \leq \bar P$, i.e. $P_k \in \Pc$.
\end{proof}

\begin{proof}[Proof of Lemma \ref{lem: uniform bdd on Pe}] In the following, we use the notations and definitions introduced in Appendix \ref{append: LQT}.1 and Theorem \ref{thm: DP solution to LQT}. 
Since $P^e$ is not influenced by $\theta_t$, we let $\theta_t=0$ for simplicity. Consider a \textit{finite-horizon} LQR problem: $ \sum_{k=0}^{N-1} (x_k^\top  Q x_k +u_k^\top  R u_k)+ x_N^\top Q_N x_N$, where $Q_N\in \Pc$. By Lemma \ref{lem: Pt in Pbar Punderline}, we have $P_k \in \Pc$. Since $P_k \to P^e$ as $k\to -\infty$ \cite{bertsekas3rddynamic}, and since $\Pc$ is a closed set \cite{zedek1965continuity},  we have $P^e \in \Pc$. Since $P^e$ and $\bar P$ are positive definite, we have $  \|P^e \|_2\leq \upsilon_{max}(\bar P)$.

	
\end{proof}

\subsection{Proof of Lemma \ref{lem: he formula}}
In the following, we will provide and prove an enhanced version of Lemma \ref{lem: he formula} with detailed characterization of the solution to the Bellman equations in Proposition \ref{prop: opt avg cost LQT}. 

\begin{proposition}[Optimal solution to average-cost LQ tracking]\label{prop: opt avg cost LQT}
	Suppose $(A, B)$ is controllable, $Q,R$ are positive definite. The optimal average cost $\lambda^e$ does not depend on the initial state $x_0$ and is equal to
	\begin{align*}
	\lambda^e =  \frac{1}{2}(A\theta-\beta^e)^\top  H^e(A\theta-\beta^e),
	\end{align*}
	where $M^e=P^e-P^eB(R+B^\top P^eB)^{-1}B^\top P^e$ and $H^e=M^e-M^eA(Q+A^\top M^eA)^{-1}A^\top M^e$.
	
	In addition, a bias function of the Bellman equations $h^e(x)+\lambda^e= \min_u (f(x)+g(u)+h^e(Ax+Bu))$ can be represented by
	$$h^e(x)=\frac{1}{2}(x-\beta^e)^\top  P^e(x-\beta^e).$$
	where $P^e=P^e(Q,R)$.
	
	The optimal controller is $$u=-K^e x +K'\beta^e$$
	 where 
	 $K^e=(R+B^\top P^eB)^{-1}B^\top P^eA$, $K'=(R+B^\top P^eB)^{-1}B^\top P^e$, and $\beta^e$ satisfies
	 \begin{align}
	     \beta^e=(P^e)^{-1}\alpha^e=F\theta \label{equ: formula for beta e}
	 \end{align}
	 where $\alpha^e= Q\theta+(A-BK^e)^\top \alpha^e$ and thus 
	 $F= (P^{e})^{-1}(I-(A-BK^e)^\top )^{-1}Q$. 
	 
\end{proposition}

\begin{proof}[Proof of Proposition \ref{prop: opt avg cost LQT}]
It is easy to see that the formulas of $\lambda^e$, $h^e(x)$, and the optimal controller are  the limits of the corresponding formulas or the limiting solutions to the corresponding iterative equations  under fixed $Q, R, \theta$ in Theorem \ref{thm: DP solution to LQT}. However, to formally prove these formulas, we still need to prove the existence of the limits, which is the focus of the following proof. In particular, the proof consists of three parts: i) verify the formula of the optimal average cost $\lambda^e$, ii) verify the formula of the  bias function $h^e(x)$, iii) verify the formula of the optimal controller.

\textit{Part i): Verify the formula of $\lambda^e$.} 
		Consider a finite horizon LQT problem:
	\begin{align*}
	\min_{x_t,u_t}& \ \ \frac{1}{2} \sum_{t=0}^{N-1} \left[ (x_t-\theta)^\top  Q(x_t-\theta)+ u_t^\top  R u_t \right] \\
	\text{s.t. } \ & \ x_{t+1}=Ax_t + Bu_t, \quad \quad t=0, \dots, N-1
	\end{align*}
	Given an initial state $x_0$, by Theorem \ref{thm: DP solution to LQT}, the total optimal cost in $N$ time steps is 
	$$J_N^*(x_0) =\frac{1}{2} (x_{0}-\beta_{0})^\top  P_{0}(x_{0}-\beta_{0})+\frac{1}{2} \sum_{k=0}^{N-1}  (A\theta  - \beta_{k+1})^\top  H_k (A\theta  - \beta_{k+1})$$
	If we can show that  $\beta_k \to \beta^e$ and $P_k\to P^e$ and $H_k\to H^e$ as $k \to -\infty$, then, consequently, we will have $\frac{1}{2} (A\theta  - \beta_{k+1})^\top  H_k (A\theta  - \beta_{k+1}) \to \frac{1}{2}  (A\theta  - \beta^e)^\top  H^e (A\theta  - \beta^e)$ as $k\to -\infty$, and  bounded $\frac{1}{2} (x_{0}-\beta_{0})^\top  P_{0}(x_{0}-\beta_{0})$ for fixed $x_0$.  Then the formula of the optimal average cost in infinite horizon can be proved by   
	\begin{align*}
	\lambda^e &=  \lim_{N \to +\infty} \left[\frac{1}{N} \min_{x_{t+1}=Ax_t + Bu_t}\left( \frac{1}{2} \sum_{t=0}^{N-1} \left( (x_t-\theta)^\top  Q(x_t-\theta)+ u_t^\top  R u_t \right) \right) \right]\\
	&=
	\lim_{N \to +\infty}\left[\frac{1}{N}\left(\frac{1}{2} (x_{0}-\beta_{0})^\top  P_{0}(x_{0}-\beta_{0})+\frac{1}{2} \sum_{k=0}^{N-1}  (A\theta  - \beta_{k+1})^\top  H_k (A\theta  - \beta_{k+1})\right)\right]\\
	&=\frac{1}{2}(A\theta-\beta^e)^\top  H^e(A\theta-\beta^e),
	\end{align*}
		Therefore, it suffices to prove $\beta_k \to \beta^e$, $P_k\to P^e$ and $H_k\to H^e$ as $k \to -\infty$. 
		
		By Proposition 4.4.1 \cite{bertsekas3rddynamic}, $P_k \to P^e$ as $k \to -\infty$. Then, $M_k \to M^e$  as $k \to -\infty$ since $M_k$ is a continuous function of $P_k$ by noticing that the matrix inverse operator is continuous when the matrix is invertible. Similarly, $H_k \to H^e$ as $k \to -\infty$ since $H_k$ is a continuous function of $M_k$. In addition, $K_k\to K^e$, and $K_k^\alpha \to K^{\alpha}$ and $K'_k \to K'$ as $k \to -\infty$ since $K_k, K_k^\alpha, K'_k$ are continuous functions of $P_k$. 
		
		To show $\beta_k \to \beta^e$, we only need to show 
		$\alpha_k \to \alpha^e$ as $k \to -\infty$ since $\beta_k = P_k^{-1}\alpha_k$. $\alpha_k$ satisfies the recursive equation $\alpha_k= Q\theta +(A-BK_{k+1})^\top \alpha_{k+1}$. Since $(A-BK_k)^\top  \to (A-BK^e)^\top $ as $k \to -\infty$ and $(A-BK^e)^\top$ is a stable matrix, by the claim below, we can show $\alpha_k \to \alpha^e$ as $k \to -\infty$. Then, the proof of Part i) is completed.
		
		\nbf{Claim: }
		\textit{If $A_t\to A$ and $A$ is stable, then the state of the system $x_{t+1}=A_t x_t +\eta$  will converge to $x^s$, where $x^s=Ax^s+\eta$, for any bounded initial value $x_0$.}
		
The proof of the claim lemma is provided at the end of this subsection.

\nit{Part ii): Verify $h^e(x)$'s formula.}
The proof is by showing the Bellman equations hold under the formulas of $h^e(x)$ and $\lambda^e$ provided in the statement of the lemma. 
\begin{align*}
 \min_u & \left[\frac{1}{2} (x-\theta)^\top  Q(x-\theta)+ \frac{1}{2}u^\top  R u + \frac{1}{2}(Ax+Bu-\beta^e)^\top  P^e(Ax+Bu-\beta^e)\right]\\
  = \ &\frac{1}{2} (x-\theta)^\top  Q(x-\theta)+ \frac{1}{2} (Ax-\beta^e)^\top M^e (Ax-\beta^e)\\
 &+\min_u \frac{1}{2}(u+ K^e x-K'\beta^e)^\top  (R+B^\top P^e B)( u +K^e x-K'\beta^e)\\
 = \ & \frac{1}{2} (x-\theta)^\top  Q(x-\theta) + \frac{1}{2}(Ax -\beta^e)^\top  M^e(Ax -\beta^e) \\
 = \ & \frac{1}{2}(A\theta-\beta^e)^\top  H^e(A\theta-\beta)+ \frac{1}{2}(x-\beta^e)^\top  P^e(x-\beta^e)
\end{align*}
where the last equality is by $Q+A^\top M^e A= P^e$,  $\beta^e = (P^e)^{-1}\alpha^e$,  $\alpha^e = Q\theta +(A-BK^e)^\top \alpha^e$ and the formulas of $K^e, M^e$.


\nit{Part iii): Verify the formula of the optimal controller.} We prove $u=-K^e x +K'\beta^e$ is the optimal controller  by showing that the average cost by implementing this controller is no more than the optimal average cost $\lambda^e$. Let $x_t,u_t$ be the state and control at $t$ by implementing the controller $u=-K^e x +K'\beta^e$.
	\begin{align*}
&\frac{1}{N}\left( \frac{1}{2} \sum_{t=0}^{N-1} \left[ (x_t-\theta)^\top  Q(x_t-\theta)+ u_t^\top  R u_t \right]\right) \\
&\leq \frac{1}{N}\left(\frac{1}{2} \sum_{t=0}^{N-1} \left[ (x_t-\theta)^\top  Q(x_t-\theta)+ u_t^\top  R u_t \right]+ \frac{1}{2}(x_N-\beta^e)^\top P^e(x_N-\beta^e)\right)\\
 &= \frac{1}{N}\left(\frac{1}{2} (x_{0}-\beta^e)^\top  P^e(x_{0}-\beta^e)+\frac{1}{2} \sum_{k=0}^{N-1}  (A\theta  - \beta^e)^\top  H^e (A\theta  - \beta^e)\right)
\end{align*}
where the last equality is by Theorem \ref{thm: DP solution to LQT}.
Taking $N\to +\infty$ on both sides, we have
\begin{align*}
\lim_{N\to +\infty}\frac{1}{N} \frac{1}{2} \sum_{t=0}^{N-1} \left[ (x_t-\theta)^\top  Q(x_t-\theta)+ u_t^\top  R u_t \right] 
\leq \frac{1}{2}  (A\theta  - \beta^e)^\top  H^e (A\theta  - \beta^e)=\lambda^e
\end{align*}
This completes the proof.
\end{proof}

\nbf{Proof of Claim:}
Define the error term $d_t =x_t-x^s$. The dynamics of $d_t$ is $d_{t+1}=Ad_t+w_t$, where $w_t=(A_t-A)(d_t+x^s)$. It suffices to show $d_t \to 0$ as $t\to +\infty$. In the following, we will first prove two facts, based on which we prove $d_t \to 0$.

\nbf{Fact 1:} \textit{Consider a stable matrix $A$ and a sequence of uniformly bounded vectors: $\| v_t \|_2 \leq D$ for all $t$. There exists a constant $c_6>0$ determined by $A$, such that, for any $k=1, 2, \dots,$ 
$$ \| \sum_{t=0}^{k-1} A^t v_t \|_2 \leq c_6 D.$$}

\nbf{Proof of Fact 1:} This is a consequence of the fact that exponential stability implies bounded-input-bounded-output stability. To see this, 	consider a system $x_{t+1}=Ax_t+u_t$ with $x_0=0$. Since $A$ is stable, the system is exponentially stable. By Theorem 9.4 \cite{hespanha2018linear}, the exponential stability implies the bounded-input-bounded-output stability. Thus, there exists $c_6$ such that $\|x_k\|_2 \leq c_6 D$ for any $k$ and  any  input sequence satisfying $\|u_t \|_2 \leq D$ for all $t$. 

For any $k\geq 0$, consider inputs $u_t=v_{k-1-t}$ for $0 \leq t \leq k-1$, then $x_k= \sum_{t=0}^{k-1} A^{t}v_t$. Since $\|u_t \|\leq D$,  we have
$    \|x_k\| \leq c_6 D$, which completes the proof.
\qed 

\vspace{16pt}

\nbf{Fact 2:} \textit{There exists a constant $D>0$, such that $\max_{t\geq 0}(\|x^s\|, \|d_t \|)\leq D$.}

\nbf{Proof of Fact 2:}
Since $A_t\to A$, for  $\epsilon_1= 1/(4c_6)$, there exists $N_1$, such that when $t\geq N_1$, $\|A_t-A\| \leq \epsilon_1$. 
Since $A$ is stable, we have $A^t  \to 0$, so for  $\epsilon_2=1/2$, there exists $N_2$, such that when $t>N_2$, $\|A^t \| \leq \epsilon_2$.

	 Let $D=\max(\|d_0\|, \dots, \|d_{N_1+N_2}\|, \|x^s\|)$. By definition, $\|d_{t}\| \leq D$ for $t\leq N_1+N_2$. We can show  $\|d_{N_1+N_2+1}\| \leq D$ in the following.
	\begin{align*}
	\|d_{N_1+N_2+1}\|&=\| A^{N_2+1}d_{N_1}+ w_{N_1+N_2} +Aw_{N_1+N_2-1}+\dots + A^{N_2}w_{N_1}\|\\
	 &\leq \|A^{N_2+1}\|_2 D + \| w_{N_1+N_2} +Aw_{N_1+N_2-1}+\dots + A^{N_2}w_{N_1}\|\\
	&  \leq \epsilon_2 D + c_6 \max_{N_1\leq k \leq N_1+N_2} \|w_k\|\\
	&  \leq \epsilon_2 D + 2c_6 \epsilon_1 D  = (1/2+1/2)D=D
	\end{align*}
	where the second inequality is by Fact 1 and the definitions of $\epsilon_2$, and the third inequality is by $w_k= (A_k-A)(d_{k}+ x^s)$, $k\geq N_1$, and the definitions of $D$ and $\epsilon_1$. 
	
	It can be shown by induction that $\|d_t \|\leq D$ for any $t \geq N_1+N_2+1$ in the same way, which completes the proof. 	\qed


\vspace{16pt}

\nbf{Prove $d_t \to0$.}
We will show that for any $\epsilon_3>0$, there exists $N_3$, such that $\|d_t\|_2\leq \epsilon_3$ when $t> N_3$. The proof is very similar to the proof of Fact 2.

Since $A_t\to A$, when $\epsilon'_1= \epsilon_3/(4c_6D)$, there exists $N'_1$, such that when $t\geq N'_1$, $\|A_t-A\| \leq \epsilon'_1$, where $D$ is defined in Fact 2. 
Since $A$ is stable, we have $A^t  \to 0$, so when $\epsilon'_2= \epsilon_3/(2D)$, there exists $N'_2$, such that when $t>N'_2$, $\|A^t \| \leq \epsilon'_2$. 
Let $N_3=N'_1+N'_2$. When $t>N_3$,
\begin{align*}
\|d_{t+1}\|&= \| A^{t-N'_1+1}d_{N'_1}+ w_t +Aw_{t-1}+\dots + A^{t-N'_1}w_{N'_1}\| \\
 &\leq \|A^{t-N'_1+1}\| D + \| w_t +Aw_{t-1}+\dots + A^{t-N'_1}w_{N'_1}\|_2\\
& \leq \epsilon'_2 D + c_6 \max_{N'_1\leq k \leq t} \|w_k\|\\
&\leq \epsilon'_2 D + 2c_6 \epsilon'_1 D = (1/2+1/2)\epsilon_3=\epsilon_3
\end{align*}
	where the second inequality is by Fact 1 and the definitions of $\epsilon_2$, and the third inequality is by $w_k= (A_k-A)(d_{k}+ x^s)$, $k\geq N_1$, $\max_{k\geq 0}(\|d_k\|, \|x^s\|)\leq D$, and the definition of $\epsilon'_1$. 

This completes the proof of the claim. \qed

\subsection{Proof of Lemma \ref{lem: bound xt*}. }
Let $x_t^*, u_t^*$ denote the optimal state and the optimal control input at time $t$ respectively. By Theorem \ref{thm: DP solution to LQT}, the optimal controller is $u_t^*=-K_t x_t^* +K_t^\alpha \alpha_t$. For ease of notation, define
$$ D_t \coloneqq A-BK_t.$$
Then, the dynamical system of $x_t^*$ can be represented as 
$$x_{t+1}^* = D_t x_t^* + BK_t^\alpha \alpha_{t+1}.$$

\textbf{Proof outline: } We will prove $x_t^*$ is bounded by three steps: 1) show that system $x_{t+1}=D_t x_t$ is exponentially stable, 2) show that $BK_t^\alpha \alpha_{t+1}$ is bounded, 3) show $x_t^*$ is bounded by the fact that exponentially stable systems are bounded-input-bounded-output stable.

\nit{Step 1: show $x_{t+1}=D_t x_t$ is exponentially stable by a Lyapunov function.}

\begin{lemma}[Exponential stability]\label{lem: exp stable}Consider dynamical system $x_{t+1}=D_t x_t$.
	Define the state transition matrix: $$\Phi(t,t_0)=D_{t-1}\cdots D_{t_0}$$
	for $t\geq t_0$, and $\Phi(t,t)=I$. 
	For any $N$, any $0\leq t_0 \leq N$ $t_0\leq t \leq N$, any $Q_t \in \Q, R_t \in \Rc, Q_N \in \Pc$, and for any $x_{t_0}$, the system is exponentially stable, i.e.
	\begin{align}
	 	\|\Phi(t,t_0) \| &\leq c_7 c_2^{t-t_0} 
	\end{align} 
	where $c_7=\sqrt{\frac{\upsilon_{max}(\bar P)}{\upsilon_{min}(\underline P)}}$, $c_2=\sqrt{1-\frac{\mu_f}{\upsilon_{max}(\bar P)}}\in [0,1)$.
\end{lemma}
\begin{proof}
We prove the exponential stability by constructing a Lyapunov function: $L(t, x_t)=x_t^\top P_t x_t$ for $t\geq 0$. 

\nbf{Claim: } \textit{For any $x_t$, the Lyapunov function satisfies
$$\upsilon_{min}(\underline P)\|x_t\|_2^2 \leq L(t,x_t)\leq \upsilon_{max}(\bar P)\|x_t\|^2, \quad
	 L(t+1, D_t x_t)-L(t,x_t)\leq -\mu_f \|x_t\|_2^2.$$
	 where $\upsilon_{max}(\cdot)$ and $ \upsilon_{min}(\cdot)$ denote the maximum and minimum eigenvalues of a matrix  respectively.}
	 
\nbf{Proof of Claim:}
	By Lemma \ref{lem: Pt in Pbar Punderline}, $P_t \in \Pc$, so
	$ \upsilon_{min}(\underline P) I_n \leq \underline P \leq P_t \leq \bar P \leq \upsilon_{max}(\bar P)I_n$.
	Thus, for any $x_t$, 
	$$\upsilon_{min}(\underline P)\|x_t\|_2^2 \leq L(t,x_t)=x_t^\top  P_t x_t\leq\upsilon_{max}(\bar P)\|x_t\|^2$$
	
Besides,
	\begin{align*}
	L(t+1, D_t x_t)-L(t,x_t)&= x_t^\top D_t^\top P_{t+1}D_t x_t -x_t^\top  P_t x_t\\
	& = x_t^\top (D_t^\top P_{t+1}D_t-P_t)x_t\\
	&= x_t^\top (-Q_t-K_t^\top R_tK_t)x_t \\
	& \leq -x_t^\top \underline Qx_t 
	= -\mu_f \|x_t\|^2 
	\end{align*}
	where the third equality is by Theorem \ref{thm: DP solution to LQT}, the first inequality and the last equality are by $Q_t+K_t^\top R_tK_t\geq Q_t \geq \underline Q=\mu_f I_n$.
	\qed
	
	By the claim above, 
	\begin{align*}
	L(t+1, x_{t+1})-L(t,x_t)\leq -\mu_f \|x_t\|_2^2 \leq -\frac{\mu_f}{\upsilon_{max}(\bar P)}L(t,x_t)
	\end{align*}
	Thus, $ L(t+1, x_{t+1}) \leq c_2^2 L(t,x_t)$ where $c_2= \sqrt{1-\frac{\mu_f}{\upsilon_{max}(\bar P)}}$. Here, $c_2$ is well-defined because $0\leq \mu_f I_n \leq \bar Q \leq \bar P\leq \upsilon_{max}(\bar P)$.

	For any $t_{t_0}$ and any $x_{t_0}$, it is easy to verify that the state $x_t$ satisfies $x_t=\Phi(t, t_{t_0})x_{t_0}$. Therefore, 
	\begin{align*}
	    \|\Phi(t, t_0)\|&= \max_{x_{t_0}\not =0} \frac{\|x_t\|}{\|x_{t_0}\|}\\
	    & \leq \max_{x_{t_0}\not =0}\sqrt{ \frac{\upsilon_{max}(\bar P)}{\upsilon_{min}(\underline P)}\frac{L(t,x_t)}{L(t_0, x_{t_0})} }\\
	    & \leq \sqrt{\frac{\upsilon_{max}(\bar P)}{\upsilon_{min}(\underline P)}} c_2^{t-t_0}
	\end{align*}
	\end{proof}

\nit{Step 2: show that $BK_t^\alpha \alpha_{t+1}$ is bounded.} We will show that 
\begin{equation}
    \| \alpha_t \| \leq \frac{c_7}{1-c_2}\upsilon_{max}(\bar P)\bar \theta\eqqcolon \bar \alpha, \qquad \| BK_t^\alpha \alpha_t\|_2 \leq \|B\|_2^2 \frac{\bar \alpha}{\mu_g}.\label{equ: bdd for alphat}
\end{equation} 

By Theorem \ref{thm: DP solution to LQT}, $\alpha_t$ satisfies the dynamical system $\alpha_t= D_t^\top \alpha_{t+1}+ Q_t\theta_t$, with initial condition $\alpha_N=Q_N\theta_N$. As a result, we can write $\alpha_t$ in terms of $\theta_s$ and the transition matrix $\Phi(t, s)$ as follows
\begin{align*}
	\alpha_t& =Q_t\theta_t + D_t^\top  Q_{t+1}\theta_{t+1}+ \dots+ D_t^\top  \dots D_{N-1}^\top  Q_N\theta_{N}\\
	&=\Phi(t,t)^\top Q_t\theta_t+ \Phi(t+1,t)^\top  Q_{t+1}\theta_{t+1}+ \dots+ \Phi(N,t)^\top  Q_N\theta_{N}
	\end{align*}
Then, the bound of $\alpha_t$ can be derived as follows.
	\begin{align*}
	\|\alpha_t\|& \leq \|\Phi(t,t)^\top \|\|Q_t\theta_t\|+ \dots+  \|\Phi(N,t)^\top \|\| Q_N\theta_{N}\|\\
	& = \|\Phi(t,t)\|\|Q_t\theta_t\|+ \dots+  \|\Phi(N,t)\|\| Q_N\theta_{N}\| \\
	& \leq c_7 \upsilon_{max}(\bar P) \bar \theta + \dots +  c_7c_2^{N-t} \upsilon_{max}(\bar P)\bar \theta  \\
	& \leq c_7 \upsilon_{max}(\bar P)\bar \theta  \frac{1}{1-c_2}=\bar \alpha
	\end{align*}
	where the second inequality is by Lemma \ref{lem: exp stable}, $Q_t \leq \bar Q \leq \bar P$ and $Q_N \in \Pc$. 
	
	Consequently, 
		\begin{align*}
	\|BK_t^\alpha \alpha_t\| & = \|B(R_t + B^\top  P_{t+1}B)^{-1} B^\top \alpha_t \|\\ &\leq \|B\|^2 \|(R_t + B^\top  P_{t+1}B)^{-1}\|\|\alpha_t\| \\
	& \leq \|B\|^2 \frac{\bar \alpha}{\mu_g} 
	\end{align*}
	

\nit{Step 3: bound $x_t^*$.}
\begin{align*}
    \|x_t^*\| & = \| \Phi(t,t)BK_{t-1}^\alpha \alpha_t+\Phi(t,t-1)BK_{t-2}^\alpha \alpha_{t-1}+\dots \Phi(t,1)BK_{0}^\alpha \alpha_1\|\\
    & \leq c_7 \|B\|^2 \frac{\bar \alpha}{\mu_g}(1+c_2 + c_2^2 +\dots)= c_7 \frac{1}{1-c_2}\|B\|^2 \frac{\bar \alpha}{\mu_g}\eqqcolon \bar x
\end{align*}

\subsection{Proof of Lemma \ref{lem: uniform bdd on beta e}}
By Theorem \ref{thm: DP solution to LQT} and \eqref{equ: bdd for alphat}, $\|\beta _k \|= \|P_k^{-1}\alpha_k\|\leq \frac{1}{\upsilon_{min}(\underline P)}\bar \alpha$ for any $k$, any $N$ and any $Q_t \in \Q, R_t \in \Rc, Q_N \in \Pc$. When $Q_t=Q \in \Q, R_t=R=\Rc$ for all $t$, $\beta_k\to \beta^e$ as $k \to -\infty$ by the proof (Part i)) of Proposition \ref{prop: opt avg cost LQT}. Thus, $\|\beta^e \| \leq \frac{1}{\upsilon_{min}(\underline P)}\bar \alpha$. 
Define $\bar \beta=\max(\bar \theta, \frac{1}{\upsilon_{min}(\underline P)}\bar \alpha)$. This completes the proof.


\section{Simulation descriptions} \label{append: simulation}

\subsection{LQT}

The experiment settings are as follows. Let $A=[0,1; -1/, 5/6 ],B=[0;1]$, $N=30$. Consider diagonal $Q_t, R_t$ with diagonal entries i.i.d.  from $\text{Unif}[1,2]$. Let $\theta_t$ i.i.d. from $\text{Unif}[-10,10]$. 
 The stepsizes of RHGD and RHTM are based on the conditions in Theorem \ref{thm: RHGM regret bdd}.  The stepsizes of RHAG can be viewed as RHTM with $\delta_c=1/l_c, \delta_y=\delta_{\omega}=\frac{\sqrt \zeta-1}{\sqrt \zeta +1}$ and $\delta_z=0$. 

\subsection{Robotics tracking}
Consider the following discrete-time counterpart of the kinematic model 
\begin{subequations}
\begin{align*}
   x_{t+1} & = x_t + \Delta t \cdot\cos \delta_t \cdot v_t \\
   {y}_{t+1} & = y_t+\Delta t\cdot \sin \delta_t \cdot v_t\\
   {\delta}_{t+1} & =\delta_t + \Delta t\cdot w_t
\end{align*}
\end{subequations}
Thus we have 
\begin{subequations}
\begin{align*}
    \delta_t & = \arctan( \frac{y_{t+1}-y_{t}}{x_{t+1}-x_t} )\\
    v_t & = \frac{1}{\Delta t}\cdot \sqrt{(x_{t+1}-x_t)^2 + (y_{t+1}-y_t)^2} \\
    w_t& = \frac{\delta_{t+1} - \delta_t}{\Delta t} = \frac{1}{\Delta t}\cdot \left[   \arctan( \frac{y_{t+2}-y_{t+1}}{x_{t+2}-x_{t+1}} ) - \arctan( \frac{y_{t+1}-y_{t}}{x_{t+1}-x_t} )                 \right]
\end{align*}
\end{subequations}

So that $(\delta_t,v_t,w_t)$ can be expressed by the state variables $(x_t,y_t)$.

In the simulation, the given reference trajectory is 
\begin{subequations}
\begin{align*}
    x^r_t & = 16\sin^3(t-6)\\
    y^r_t& = 13\cos(t)-5 \cos(2t-12)-2\cos(3t-18)-\cos(4t-24)
\end{align*}
\end{subequations}
 As for the objective function, we set the cost coefficients as
\begin{align*}
    c_t = \begin{cases}
    0, & t = 0\\
    1, & \text{otherwise}
    \end{cases}\qquad
        c_t^v = \begin{cases}
    0, & t = N\\
    15\Delta t^2, & \text{otherwise}
    \end{cases}\qquad     c_t^w = \begin{cases}
    0, & t = N\\
    15\Delta t^2, & \text{otherwise}
    \end{cases}
\end{align*}
The discrete-time resolution for online control is 0.025 second, i.e., $\Delta t =0.025s$. When implementing each control decision, a much smaller time resolution of $0.001s$ is used to simulate the real motion dynamics of the robot.

\end{document}